\documentclass[11pt]{article}

\usepackage{amssymb, amsfonts, amsxtra, amsmath}
\usepackage{mathrsfs}
\usepackage{amsthm}
\usepackage{latexsym}
\usepackage[all]{xy}
\usepackage{graphics}
\usepackage{makeidx}
\usepackage{rotating}
\usepackage{bbold}
\usepackage{syntonly}
\usepackage{nomencl}
\usepackage{pdfpages}

\usepackage[normalem]{ulem}

\usepackage{hyperref}

\usepackage{parskip}
\makeatletter 
\def\thm@space@setup{%
 \thm@preskip=\parskip \thm@postskip=0pt
}
\def\th@remark{%
  \thm@headfont{\itshape}%
  \normalfont 
  \thm@preskip\parskip \thm@postskip=0pt
}
\makeatother

\newtheorem{Theorem}{Theorem}[section]
\newtheorem{Def}[Theorem]{Definition}
\newtheorem{Lem}[Theorem]{Lemma}
\newtheorem{Prop}[Theorem]{Proposition}
\newtheorem{Cor}[Theorem]{Corollary}

\newtheorem{Rem}[Theorem]{Remark}

\DeclareMathOperator{\Tr}{\mathrm{Tr}}

\newcommand{\C}{\mathbb{C}}
\newcommand{\R}{\mathbb{R}}

\newcommand{\Hsp}{\mathcal{H}}

\newcommand{\Gg}{\mathbb{G}}

\newcommand{\Xx}{\mathbb{X}}

\newcommand{\Ww}{\mathcal{W}}
\newcommand{\Vv}{\mathcal{V}}
\newcommand{\mWw}{\mathscr{W}}

\newcommand{\mfsl}{\mathfrak{sl}}

\newcommand{\opp}{\mathrm{op}}
\newcommand{\cop}{\mathrm{cop}}

\newcommand{\Ad}{\mathrm{Ad}}

\newcommand{\id}{\mathrm{id}}

\newcommand{\vNtimes}{\bar{\otimes}}

\newcommand{\acnabla}{\nabla\!\!\!{^\shortmid}\,\,}

\title{$I$-factorial quantum torsors}

\author{Kenny De Commer\thanks{Vakgroep wiskunde, Vrije Universiteit Brussel (VUB), B-1050 Brussels, Belgium, email: {\tt kenny.de.commer@vub.ac.be}} \thanks{Partially supported by the FWO grant G.0251.15N and the grant H2020-MSCA-RISE-2015-691246-QUANTUM DYNAMICS}}
\date{}

\begin{document}
\maketitle
\date

\begin{abstract}\noindent In an earlier paper of the author, locally compact quantum torsors were defined for locally compact quantum groups, putting into the analytic framework the theory of Galois objects for Hopf algebras. Such quantum torsors allow to deform the given quantum group, providing a generalization of the $2$-cocycle twisting procedure. It was also shown that a quantum torsor can be constructed from an action of the dual quantum group on a type $I$-factor. In this paper, we study quantum torsors which are themselves type $I$-factors. These $I$-factorial quantum torsors turn out to have a nice duality theory. We illustrate the general theory with the example of the Heisenberg double.
\end{abstract}

\section*{Introduction}

The theory of \emph{locally compact quantum groups} \cite{KV00,KV03} provides a vast generalization of the classical theory of locally compact groups. Being steeped in the theory of von Neumann algebras and Tomita-Takesaki theory, it is the proper setting in which to study quantum symmetries such as arising for example from subfactor theory \cite{EnN96,Eno98,Vae01}. One of the main attributes of the theory is the existence of a \emph{generalized Pontryagin duality theory}, allowing for a uniform treatment of many classical group theoretical results and constructions. 

By now, many construction methods for locally compact quantum groups are known. We mention for example
\begin{itemize}
\item the bicrossed product construction \cite{Maj91,VVa03},
\item the double crossed product construction \cite{WoP90,BV05},
\item $2$-cocycle twisting \cite{EV96,DeC11}.
\end{itemize}

Apart from this, many specific examples have been studied, especially within the context of \emph{compact} quantum groups \cite{Wor98}, such as \emph{$q$-deformations of semi-simple compact Lie groups} \cite{LS91} and \emph{universal quantum groups} \cite{VDW96}. Within the non-compact setting, we mention among others the quantizations of low-dimensional Lie groups \cite{Wor01,Wor91,KK03,VVa02} and the interesting developments in \cite{ByT03,Ip15} concerning quantizations of split real Lie groups (which have however not been formally put yet within the framework of locally compact quantum groups).

In this article, we will be concerned with the generalization of the $2$-cocycle twisting technique as defined in \cite{DeC11}. More precisely, in that article a specific class of quantum group actions  on von Neumann algebras was studied, consisting of those actions which are in a sense ``free, transitive and proper''. Such actions will be called \emph{quantum torsors}. It was shown how from such a quantum torsor can be constructed a new locally compact quantum group by a twisting procedure.

The theory in \cite{DeC11} was only partially developed, up to the point where the applications to the theory of $2$-cocycle twisting could be presented. The full story was presented in the Ph.D.~ thesis of the author, parts of which however remained unpublished. In the meantime, Enock generalized these results to the broader setting of \emph{measured quantum groupoids} \cite{Eno12}. Since it will be important to have these results available in the technically easier setting of locally compact quantum groups, we will spend some time introducing the relevant terminology and fundamental statements.      However, our proofs are devised as to arrive at our end result in the quickest way possible by using results already explicitly stated in the literature. Although a direct, first principles approach would certainly be possible, we do not feel this would be worth the effort, as many of the needed techniques are well-known already and would offer nothing interesting or new. 

It was also shown in \cite{DeC11} how a quantum torsor for a quantum group can be constructed from an action of its dual on a type $I$-factor, called \emph{projective representation}. Moreover, this turns out to produce a $1$-to-$1$-correspondence between quantum torsors for a quantum group and projective representations of its dual (up to the appropriate notion of equivalence). 

The \emph{main object} of this paper will be the study of quantum torsors which are \emph{at the same time} projective representations. These quantum torsors will be called \emph{$I$-factorial quantum torsors}. This provides for a particularly nice situation, as it turns out that the associated projective representation of the dual is again a $I$-factorial quantum torsor, thus allowing one to twist both the original quantum group as well as its dual. This situation is not as uncommon as it may seem on first sight, and an interesting class of examples will be presented in the article \cite{DeC16}, for which this paper sets the stage. Here, however, we will content ourselves with the simpler example of the \emph{Heisenberg double}. 

The precise contents of this paper are as follows. 

In the \emph{first section}, we recall some basic facts from the theory of locally compact quantum groups in the von Neumann algebraic setting. In the \emph{second section}, we introduce the notion of \emph{Galois object}, and recall some of the main results of \cite{DeC11}. In the \emph{third section}, we define \emph{co-linking quantum groupoid von Neumann algebras}, and prove that they provide very special examples of \emph{measured quantum groupoids} \cite{Les07,Eno08}. We then show, following partly \cite{Eno12}, how any Galois object can be completed, in an essentially unique way, into a co-linking quantum groupoid von Neumann algebra. In the \emph{fourth} section we consider the \emph{reduction} of a Galois object with respect to a (dual) quantum subgroup, providing for concrete examples a useful technique to break down a complicated Galois object into simpler constituents. These results were proven in the Ph.D. thesis of the author, but have not been published before. In the  \emph{fifth section} we introduce \emph{$I$-factorial Galois objects}, and prove our main theorem, Theorem \ref{TheoDualMain}.  Finally, in the \emph{sixth} section we illustrate the general theory with the \emph{Heisenberg double} example. 

\emph{Notation}

We denote by $\otimes_{\C}$ the tensor product of vector spaces, by $\otimes$ the tensor product of Hilbert spaces, and by $\vNtimes$ the spatial tensor product of von Neumann algebras. We further denote by \[\Sigma: V\otimes_{\C} W\rightarrow W\otimes_{\C} V,\quad v\otimes w \mapsto w\otimes v\] the flip map for vector spaces as well as Hilbert spaces, while for algebras $A$,$B$ we will rather use the notation \[\varsigma: A\otimes_{\C} B\rightarrow B\otimes_{\C} A,\quad a\otimes b\mapsto b\otimes a,\]and similarly for tensor products of von Neumann algebras. 

For $\Hsp$ a Hilbert space we denote 
\[
\omega_{\xi,\eta}(x) = \langle \xi,x\eta\rangle,\qquad x\in B(\Hsp),\xi,\eta\in \Hsp.
\] 
We also make frequent use of the leg numbering notation as is common in quantum group theory. For example, if $X\in B(\Hsp^{\otimes 2})$, then 
\[
X_{13} = (\id\otimes \varsigma)(X\otimes 1) = (1\otimes \Sigma)(X\otimes 1)(1\otimes \Sigma) \in B(\Hsp^{\otimes 3}).
\] 

When $M$ is a von Neumann algebra, we denote by $M^+$ its positive cone, by $M_*$ its pre-dual, by $M_*^+$ the positive cone of its predual, and by $\mathcal{S}_*(M)$ its space of normal states. For $\omega \in M_*$ we denote 
\[
\overline{\omega}(x) = \overline{\omega(x^*)}.
\] 
We denote by $L^2(M) = (L^2(M),\pi_M,J_M,\mathfrak{P})$ a \emph{standard form} of $M$ \cite[Chapter IX, Definition 1.13]{Tak2}, where $\pi_M$ is the standard representation of $M$ on $L^2(M)$, where $J_M = J_M^* = J_M^{-1}$ is the \emph{modular conjugation} and where $\mathfrak{P}$ is the positive cone. In practice we will suppress the notation $\pi_M$ and view $M\subseteq B(L^2(M))$. 

We also use standard notation for weight theory: if $\varphi$ is a normal, semi-finite, faithful (nsf) weight on a von Neumann algebra $M$, we denote 
\[
\mathscr{N}_{\varphi} = \{x\in M\mid \varphi(x^*x)<\infty\},\quad \mathscr{M}_{\varphi} = \mathscr{N}_{\varphi}^*\mathscr{N}_{\varphi},\quad \mathscr{M}_{\varphi}^+ = \mathscr{M}_{\varphi}\cap M^+,
\] 
and we write $\sigma_t^{\varphi}$ for the modular automorphism group. We denote by
\[
\Lambda_{\varphi}: \mathscr{N}_{\varphi} \rightarrow L^2(M)
\]
the canonical \emph{GNS-map}, and by $\nabla_{\varphi}$ the modular operator, so that 
\[
\Lambda_{\varphi} \circ \sigma_t^{\varphi} = \nabla_{\varphi}^{it}\circ \Lambda_{\varphi}.
\]

\section{Quantum group von Neumann algebras}

\begin{Def}\cite[Definition 1.1]{KV03} A \emph{quantum group von Neumann algebra} consists of a von Neumann algebra $M$ together with a coassociative unital normal $*$-homomorphism \[\Delta: M \rightarrow M\vNtimes M\]  for which there exist normal, semi-finite, faithful (nsf) weights \[\varphi,\psi: M^+ \rightarrow [0,+\infty]\] such that for all $x\in M^+$ and all $\omega\in \mathcal{S}_*(M)$ \[\varphi((\omega\otimes \id)\Delta(x)) = \varphi(x),\qquad \psi((\id\otimes \omega)\Delta(x)) = \psi(x).\] These properties are called respectively \emph{left-invariance} and \emph{right-invariance}.
\end{Def} 

\begin{Rem} 
\begin{enumerate}
\item In \cite[Definition 1.1]{KV03}, these objects are called rather `von Neumann algebraic quantum groups'. We prefer the above terminology as it refers more directly to `an algebra associated to the quantum group' instead of vice versa. Also the terminology Hopf-von Neumann algebra is sometimes used for the above structure.
\item As shown in \cite{KV03}, one can associate to $(M,\Delta)$ a unique (reduced) C$^*$-algebra with an appropriate comultiplication and weights, forming a \emph{quantum group C$^*$-algebra} (adapting terminology of \cite{KV00}). This C$^*$-algebraic structure then completely remembers the original quantum group von Neumann algebra. This justifies the interpretation of a quantum group von Neumann algebra $(M,\Delta)$ as being associated to a \emph{locally compact quantum group}. Since the use of the latter terminology will not be particularly helpful in this paper and can sometimes cause confusion, we will refrain from using it explicitly, having however used it in the introduction for the sake of intuition.
\item In \cite{KV03}, the main results are proven by copying techniques from the C$^*$-algebra setting \cite{KV00}. For a more direct von Neumann algebra oriented approach, see \cite{VD14}. The latter does not contain any new results with respect to \cite{KV00,KV03}, but contains some simplifications of and new approaches to the original results, and is sometimes more suitable as a direct reference. 
\item The invariance condition is stated in a slightly stronger form than needed, see \cite[Proposition 3.1]{KV03}.
\end{enumerate}
\end{Rem} 

The nsf weights $\varphi,\psi$ are unique up to scalars \cite[Theorem 2.5]{VD14}, and we usually suppose them fixed from the outset. We then denote 
\[
M = (M,\Delta) = (M,\Delta,\varphi,\psi)
\] 
as a shorthand. 

Let $(M,\Delta)$ be a general quantum group von Neumann algebra. We write the modular conjugation $J  = J_M$ and further use the shorthand notation
\[
\Lambda = \Lambda_{\varphi},\quad \Gamma = \Lambda_{\psi}, \quad \sigma_t = \sigma^{\varphi}_t,\qquad  \sigma'_t = \sigma^{\psi}_t,\quad \nabla^{it} = \nabla_{\varphi}^{it},\quad \acnabla^{it} = \nabla_{\psi}^{it}.
\]
\begin{Rem}
Note that the map $\Gamma$ used in \cite{KV00,KV03} is not quite the same as the one   above, since in our notation $\Lambda_{\psi}$ denotes the standard GNS-map, while the map $\Gamma$ in \cite{KV00,KV03} is constructed more explicitly. The difference is however only by a unimodular scalar - we will come back to this in a moment. 
\end{Rem} 

Associated to $(M,\Delta)$ we have the \emph{unitary left and right regular corepresentations} \[W \in M\vNtimes B(L^2(M)),\qquad V\in B(L^2(M))\vNtimes M,\] uniquely determined by the fact that for all $\omega \in B(L^2(M))_*$ one has 
\[(\omega \otimes \id)(W^*) \Lambda(x) = \Lambda((\omega\otimes \id)\Delta(x)),\qquad x\in \mathscr{N}_{\varphi},\]
\[(\id \otimes \omega)(V) \Gamma(x) = \Gamma((\id\otimes \omega)\Delta(x)),\qquad x\in \mathscr{N}_{\psi}.\] They are multiplicative unitaries in the sense that $W$ and $V$ are unitaries satisfying \[W_{12}W_{13}W_{23} = W_{23}W_{12},\qquad V_{12}V_{13}V_{23} = V_{23}V_{12},\] see \cite[Theorem 1.2]{KV03}. There then exists a unique $\sigma$-weakly continuous one-parameter group $(\tau_t)_{t\in\R}$ of normal $*$-automorphisms of $M$ and a unique involutive anti-automorphism $R$ of $M$ such that, writing $\tau_z$ for the analytic extension of $\tau_t$ to $z\in \C$ and \[S: \mathscr{D}(\tau_{-i/2}) \subseteq M \rightarrow M,\quad x \mapsto R(\tau_{-i/2}(x)),\] one has $(\id\otimes \omega)(W)$ and $(\omega\otimes \id)(V)$ in $\mathscr{D}(\tau_{-i/2})$ for all $\omega \in B(L^2(M))_*$ and \[S((\id\otimes \omega)(W)) = (\id\otimes \omega)(W^*),\qquad S((\omega\otimes \id)(V)) = (\omega\otimes \id)(V^*),\] see \cite[Proposition 1.24]{VD14}. One has, see e.g. \cite[Theorem 2.3 and Theorem 2.10]{VD14}, \begin{equation}\label{EqTau}\Delta \circ\tau_t = (\tau_t\otimes \tau_t)\circ \Delta = (\sigma_t\otimes \sigma_{-t}')\circ \Delta\end{equation} and \begin{equation}\label{EqR} \Delta \circ R = (R\otimes R)\circ \Delta^{\opp},\end{equation} where $\Delta^{\opp}(x) = \varsigma \Delta(x)$. One calls $(\tau_t)_{t\in \R},R,S$ respectively the \emph{scaling group}, \emph{unitary antipode} and \emph{antipode} of $(M,\Delta)$. 

By the above anti-comultiplicativity of $R$, it is usual to normalize $\psi$ in function of $\varphi$ by putting \[\psi = \varphi\circ R,\] and we assume this in what follows. It can then be shown that there exists a unique scalar $\nu>0$ and unique invertible operator $\delta>0$ affiliated with $M$ such that the Connes cocycle derivative of $\psi$ with respect to $\varphi$ satisfies \[(D\psi : D\varphi)_t = \nu^{it^2/2}\delta^{it},\] see the proof of \cite[Theorem 2.11]{VD14}. One calls $\nu$ the \emph{scaling constant} and $\delta$ the \emph{modular element} of $(M,\Delta)$. They are independent of the normalisation of $\varphi$. One has that the $\delta^{it}$ are grouplike, \[\Delta(\delta^{it}) = \delta^{it}\otimes \delta^{it},\] and satisfy \[\sigma_t(\delta^{is}) = \nu^{ist}\delta^{is}.\]

\begin{Rem} It follows from the last commutation relation that $\Gamma$ is determined explicitly by \[\Gamma(x) = \nu^{-i/4}\Lambda(x\delta^{1/2}),\] where $x$ ranges over those $x\in M$ such that the composition $x\delta^{1/2}$ closes to a bounded operator inside $\mathscr{N}_{\varphi}$. Hence our $\Gamma$ differs from the $\Gamma$ in \cite{KV00,KV03} by the scalar $\nu^{-i/4}$. 
\end{Rem} 

Associated to $(M,\Delta)$ are a variety of related quantum group von Neumann algebras.

First of all, we can flip the coproduct, \[\Delta^{\opp} = \varsigma \circ \Delta,\] and we will write this as $M^{\cop}= (M,\Delta^{\opp})$. We endow this with the respective left and right invariant nsf weights $\psi$ and $\varphi$, so that the associated multiplicative unitaries become \[W^{\cop} = \Sigma V^*\Sigma,\qquad V^{\cop} = \Sigma W^*\Sigma,\] with $\Sigma$ the flip map on $L^2(M)\otimes L^2(M)$. 

On the other hand, we can also flip the product, which by Tomita theory can be realized by taking the commutant, \[M^{\prime}  =  JMJ.\] We can endow $M^{\prime}$ with the coproduct \[\Delta^{\prime}(x) = (J\otimes J)\Delta(JxJ)(J\otimes J),\] so that invariant weights are given by \[\varphi^{\prime}(x) = \varphi(Jx^*J),\qquad \psi^{\prime}(x) = \psi(Jx^*J).\] Choosing as the GNS-maps \[\Lambda^{\prime}(x) = J\Lambda(JxJ),\qquad \Gamma^{\prime}(x) = J\Gamma(JxJ),\] we have the associated multiplicative unitaries \[W^{\prime} = (J\otimes J)W(J\otimes J),\qquad V^{\prime} = (J\otimes J)V(J\otimes J).\] There is an isomorphism of quantum group von Neumann algebras \begin{equation}\label{EqIsoOpCop}\theta: M^{\cop} \rightarrow M^{\prime},\quad x \mapsto JR(x)^*J.\end{equation}

We now treat the duality theory for quantum group von Neumann algebras.

\begin{Theorem}\cite[Definition 1.5]{KV03} The $\sigma$-weak closure of \[\{(\omega\otimes \id)(W)\mid \omega \in M_*\}\] is a von Neumann algebra, and defines a quantum group von Neumann algebra by means of the comultiplication \[\hat{\Delta}(x) = \Sigma W(x\otimes 1)W^*\Sigma.\]
\end{Theorem}

\begin{Def} We define $M^{\wedge} = \hat{M}$ to be the quantum group von Neumann algebra given by the von Neumann algebra \[M^{\wedge} =\hat{M} = \{(\omega\otimes \id)(W)\mid \omega \in M_*\}''\] with the above coproduct \[\Delta^{\wedge}(x) = \hat{\Delta}(x) = \Sigma W(x\otimes 1) W^* \Sigma.\] We call $(\hat{M},\hat{\Delta})$ the \emph{Pontryagin dual} of $M= (M,\Delta)$. 
\end{Def} 

In the following, we will denote \[\lambda(\omega) = (\omega\otimes \id)(W) \in \hat{M},\qquad \omega \in M_*.\]
Let $\mathscr{I}$ be the set of $\omega \in M_*$ for which there exists a vector $\xi_{\omega}\in L^2(M)$ such that \[\omega(y^*) =\langle \Lambda(y),\xi_{\omega}\rangle,\qquad \forall y \in \mathscr{N}_{\varphi}.\] The vector $\xi_{\omega}$ is uniquely determined, and one can then show that the $\sigma$-strong-norm closure $\hat{\Lambda}$ of \[\lambda(\mathscr{I}) \subseteq \hat{M}\rightarrow L^2(M),\quad \lambda(\omega) \mapsto \xi_{\omega}\] is the GNS-map of a uniquely determined left invariant nsf weight $\hat{\varphi}$ for $M^{\wedge}$, see e.g.~ \cite[Theorem 3.13]{VD14}. We can hence identify in this way $L^2(\hat{M}) = L^2(M)$. 

We adorn the associated modular and scaling data of $\hat{M}$ with $\,\hat{\;}\,$. Then by \cite[Proposition 2.15]{KV03} the multiplicative unitaries of $\hat{M}$ are given by \begin{equation}\label{EqFormDual}\hat{W} = \Sigma W^*\Sigma = (\hat{J}\otimes \hat{J})V(\hat{J}\otimes \hat{J}),\qquad \hat{V} = W^{\prime} =  (J\otimes J)W(J\otimes J).\end{equation} We also note, see e.g.~ \cite[Proposition 2.1]{KV03}, that the unitary antipode $R$ of the original quantum group von Neumann algebra $(M,\Delta)$ is implemented by $\hat{J}$, \begin{equation}\label{EqRhatJ} R(x) = \hat{J}x^*\hat{J},\qquad x\in M.\end{equation} The scaling constant $\hat{\nu}$ equals $\nu^{-1}$, and pops up in the commutation relation between $J$ and $\hat{J}$, \[J\hat{J} = \nu^{-i/4} \hat{J}J.\] It follows that we can make a self-adjoint unitary \[\Theta = \nu^{i/8} J\hat{J},\] so that $\Theta$ provides a spatial implementation of the isomorphism $\theta$ in \eqref{EqIsoOpCop}.

We have the following commutation relations between the multiplicative unitaries and the modular conjugations, see \cite[Corollary 2.2]{KV03}: \[(\hat{J}\otimes J)W(\hat{J}\otimes J) = W^*,\qquad (J\otimes \hat{J})V(J\otimes \hat{J}) = V^*.\]

By \eqref{EqFormDual}, it follows that \[M^{\wedge\,\prime} = \{(\id\otimes \omega)(V) \mid \omega \in M_*\}''.\] We will denote \[\rho(\omega) = (\id\otimes \omega)(V),\qquad \omega \in M_*.\]

\begin{Theorem}\cite[Theorem 3.18]{VD14} The following \emph{Pontryagin biduality} holds:  \[M^{\wedge\wedge} = M.\] 
\end{Theorem}

We can write here an actual equality since by construction the above von Neumann algebras are all concretely implemented on the same Hilbert space $L^2(M)$. 

We will in the following mostly need the commutant $(M^{\wedge\prime},\Delta^{\wedge\prime}) = (\hat{M}',\hat{\Delta}')$ of the dual. Also the operation $M \mapsto M^{\wedge\prime}$ is involutive, but only after twisting with $\Theta$, \[\Theta M^{\wedge\prime\wedge\prime}\Theta = M,\qquad (\Theta\otimes \Theta)\Delta^{\wedge\prime\wedge\prime}(x)(\Theta \otimes \Theta) = \Delta(\Theta x \Theta).\] We also note for further reference that, following the above constructions, the left and right multiplicative unitaries of $M^{\wedge\prime\wedge\prime}$ are given by \[W_{\Theta} = (\Theta\otimes \Theta)W(\Theta\otimes \Theta),\qquad V_{\Theta} = (\Theta\otimes \Theta)V(\Theta\otimes \Theta),\] in agreement with the above isomorphism. 

Finally, we will also need the notion of \emph{quantum group von Neumann subalgebra}.

\begin{Def}\cite[Definition 2.9]{VVa02} We call a quantum group von Neumann algebra $(M_1,\Delta_1)$ a \emph{quantum group von Neumann subalgebra} of $(M,\Delta)$ if we are given a unital, normal inclusion of von Neumann algebras $M_1\subseteq M$ such that $\Delta_1 = \Delta_{\mid M}$. 
\end{Def} 

We recall the following theorem.

\begin{Theorem}\cite[Proposition 10.5]{BV05}\label{TheoSubvn} Let $(M,\Delta)$ be a quantum group von Neumann algebra. A von Neumann subalgebra $M_1\subseteq M$ is a quantum group von Neumann subalgebra (with respect to the restriction of $\Delta$) if and only if \[\Delta(M_1) \subseteq M_1\vNtimes M_1,\qquad R(M_1) = M_1,\quad \tau_t(M_1) = M_1,\quad \forall t\in \R.\]
\end{Theorem} 

\section{Coactions and Galois objects}

Let us fix a quantum group von Neumann algebra $(M,\Delta)$. 

\begin{Def}
We call \emph{right coaction} of $(M,\Delta)$ on a von Neumann algebra $N$ any unital normal $*$-homomorphism \[\alpha: N\rightarrow N\vNtimes M\] such that \[(\id\otimes \Delta)\alpha= (\alpha\otimes \id)\alpha.\] 
\end{Def} 

Similarly, one defines left coactions $\gamma:N\rightarrow M\vNtimes N$. Any right coaction $(N,\alpha)$ of $(M,\Delta)$ then determines a left coaction $(N,\alpha^{\opp})$ of $(M,\Delta^{\opp})$ by \[\alpha^{\opp}: N \rightarrow M\vNtimes N,\quad x \mapsto \varsigma \alpha(x),\] allowing to transfer statements concerning right coactions to corresponding ones for left coactions. 

We will in the following be mainly concerned with right coactions. 

\begin{Def} The \emph{subalgebra of coinvariants} for a right coaction $(N,\alpha)$ is the von Neumann subalgebra  \[N^{\alpha} = \{x\in N\mid \alpha(x) = x\otimes 1\}\subseteq N.\] We call $\gamma$ \emph{ergodic} if $N^{\alpha}= \C$.
\end{Def}

\begin{Def} The  \emph{crossed product} von Neumann algebra for a right coaction $(N,\alpha)$ is the von Neumann algebra \[N\rtimes M = N\rtimes_{\alpha} M = \{(1\otimes x)\alpha(y)\mid x\in \hat{M}',y\in N\}'' \subseteq N\vNtimes B(L^2(M)).\] 
\end{Def}

One has on $N\rtimes M$ a \emph{dual right coaction} $\hat{\alpha}$ of $\hat{M}'$, determined by \begin{equation}\label{EqDualCoaction}\hat{\alpha}(z) = W_{\Theta,23}(z\otimes 1)W_{\Theta,23}^*.\end{equation} This definition entails for $x\in N$ and $y\in \hat{M}'$ 
\[\hat{\alpha}(\alpha(x)) = 1\otimes \alpha(x),\qquad \hat{\alpha}(1\otimes y) = 1\otimes \hat{\Delta}'(y).\] 
 
Let us recall the following \emph{biduality result}, cf. \cite[Theorem 2.6]{Vae01} for the left handed version.

\begin{Theorem}\label{TheoBiduCoac}  Considering \[(N\rtimes_{\alpha} M)\rtimes_{\hat{\alpha}} \hat{M}' \subseteq N\vNtimes B(L^2(M)\otimes L^2(M)),\] one has an isomorphism of von Neumann algebras \begin{equation}\label{EqDefChiRtimes}\chi_{\rtimes}: N \vNtimes B(L^2(M)) \cong 
(N\rtimes M)\rtimes \hat{M}',\qquad x \mapsto V^*_{23}x_{12}V_{23}.\end{equation} In particular, one has \[\alpha(x) \mapsto \alpha(x)\otimes 1,\quad 1\otimes y \mapsto 1\otimes \hat{\Delta}'(y),\quad 1\otimes z \mapsto 1\otimes 1 \otimes z\] for $x\in N,y\in \hat{M}'$ and $z \in M^{\wedge\prime\wedge\prime}$. 

Moreover, this isomorphism satisfies the equivariance condition 
\begin{equation}\label{EqEqBidual}
(\chi_{\rtimes}^{-1}\otimes \Ad(\Theta))(\alpha^{\wedge\wedge}(x)) = \Sigma_{23}W_{23}(\alpha\otimes \id)(\chi_{\rtimes}^{-1}(x))W_{23}^*\Sigma_{23}
\end{equation}
for $x\in (N\rtimes M) \rtimes \hat{M'}$.
\end{Theorem}

\begin{Def} A right coaction $(N,\alpha)$ is called \emph{integrable} if the set \[\{x \in N^+\mid \exists y\in N^+,\forall \omega \in M_*^+, \varphi((\omega\otimes \id)\alpha(x)) = \omega(y)\}\] has $\sigma$-weakly dense linear span in $N$.
\end{Def}

If $x$ lies in the above set, the element $y$ is uniquely determined and lies in $N^{\alpha}$. We write \[y = (\id\otimes \varphi)\alpha(x)\in N^{\alpha}.\] For example, any dual coaction $\hat{\alpha}$ is integrable. See again \cite[Section 2]{Vae01}.

Choosing a fixed state $\omega$ on $N^{\alpha}$, we can then define an nsf weight \begin{equation}\label{EqVarphiN}\varphi_N(x) = \omega((\id\otimes \varphi)\alpha(x)),\qquad x\in N^+\end{equation} on $N$, with associated GNS-map $\Lambda_N$. We have in this case the following expression for the \emph{standard implementing unitary} for $\alpha$.

\begin{Lem}\cite[Proposition 2.4]{Vae01} Let $\alpha$ be an integrable right coaction. There exists a unique unitary $U \in B(L^2(N))\vNtimes M$ such that \[(\id\otimes \omega_{\eta,\xi})(U)\Lambda_N(x) = \Lambda_N((\id\otimes \omega_{\eta,\delta^{-1/2}\xi})\alpha(x))\] for all $\eta \in L^2(M)$, $\xi\in \mathscr{D}(\delta^{-1/2})$ and $x\in \mathscr{N}_{\varphi_N}$. This $U$ satisfies \[(\id\otimes \Delta)(U) = U_{12}U_{13}.\] 
\end{Lem} 
We then write 
\[
\rho_N(\omega) = (\id\otimes \omega)(U).
\]

In the case of an integrable coaction, we can also present $N\rtimes M$ on $L^2(N)$.

\begin{Theorem}\label{TheoIntImp}\cite[Theorem 5.3]{Vae01} If $(N,\alpha)$ is an integrable right coaction of $(M,\Delta)$, there exists a unique normal unital $*$-homomorphism 
\begin{equation}\label{EqRepNcrossM}
\pi_{\rtimes}: N\rtimes M \rightarrow B(L^2(N))
\end{equation}
such that 
\[
\alpha(x)(1\otimes \rho(\omega)) \mapsto x\rho_N(\omega)
\] for $x\in N$ and $\omega\in M_*$.  Moreover, the range of this map equals $J_N (N^{\alpha})' J_N$. 
\end{Theorem}

In particular, we will write \begin{equation}\label{EqRephatM} \hat{\pi}': \hat{M}'\rightarrow B(L^2(N)),\qquad 1\otimes x \mapsto \hat{\pi}'(x) = \pi_{\rtimes}(1\otimes x),\qquad x\in \hat{M}'.\end{equation}

We will be interested in right coactions $(N,\alpha_N)$ which are both ergodic and integrable. In this case, the nsf weight $\varphi_N$ in \eqref{EqVarphiN} is uniquely determined (choosing on $\C$ the canonical functional $\omega(z) = z$), and we will call it (at the moment purely by analogy) the \emph{left invariant nsf weight} on $N$. It is easily seen that there then exists a unique coisometry \[\Ww: L^2(N)\otimes L^2(M)\rightarrow L^2(N)\otimes L^2(N),\] called the \emph{Galois map}, such that for all $\omega \in B(L^{2}(N))_*$ \[(\omega\otimes \id)(\mathcal{W}^*) \Lambda_{N}(x) = \Lambda_{N}((\omega\otimes \id)\alpha_N(x)),\quad x\in \mathscr{N}_{\varphi_{N}}.\]  Note that in \cite{DeC11} the notation $\widetilde{G} = \Sigma \mathcal{W}^*\Sigma$ was used.

\begin{Def} We say that a coaction $(N,\alpha_N)$ is a \emph{right Galois object} if $\alpha$ is ergodic and integrable with \emph{unitary} Galois map $\Ww$.\end{Def}

Interpreting $M$ as a quantum group function von Neumann algebra $L^{\infty}(\Gg)$, one can envision $N$ as a function von Neumann algebra $L^{\infty}(\Xx)$, with $\Gg$ acting on the locally compact space $\Xx$ in a free, transitive and proper way (and with some quasi-invariant measure on $\Xx$ fixed). In other words, $\Xx$ can be seen as a ``\emph{locally compact quantum torsor}'' for $\Gg$. See also \cite{Sch04} for an overview of this concept within the purely algebraic setting. 

For $(N,\alpha)$ a Galois object, it was shown in \cite[Section 4]{DeC11} that \begin{equation}\label{EqWInN}\Ww \in N\vNtimes B(L^2(M),L^2(N)),\end{equation} that $\Ww$ satisfies the \emph{hybrid pentagon equation} \begin{equation}\label{EqHybPent} \Ww_{12}\Ww_{13}W_{23} = \Ww_{23}\Ww_{12},\end{equation} and that $\Ww$ implements the coaction $\alpha$ in the sense that \begin{equation}\label{EqImplGalAl}\alpha(x) = \Ww^*(1\otimes x)\Ww,\qquad x\in N.\end{equation} In particular, we find \begin{equation}\label{ActAlphN} (\alpha\otimes \id)\Ww= \Ww_{13}W_{23}.\end{equation} We also recall from \cite[Lemma 4.2]{DeC11} that, with $\hat{\pi}'$ as in \eqref{EqRephatM},  \begin{equation}\label{EqGalImpCross2} \Ww(1\otimes x) = (1\otimes \hat{\pi}'(x))\Ww,\qquad x\in \hat{M}'.\end{equation}

In practice, one can sometimes avoid working directly with the Galois unitary. Indeed, from \cite[Theorem 3.1]{DeC11} one sees that an integrable ergodic coaction $(N,\alpha)$ defines a Galois object if and only if the representation $\pi_{\rtimes}:N\rtimes M \rightarrow B(L^2(N))$ in Theorem \ref{TheoIntImp} is faithful (also known as \emph{saturatedness} of the coaction) and hence an isomorphism. As a type $I$-factor is simple as a von Neumann algebra, we obtain the following corollary.

\begin{Cor}\label{CorTypI} A right coaction $(N,\alpha)$ is a Galois object if and only if $\alpha$ is ergodic and integrable with $N\rtimes M$ a type $I$-factor, in which case we have a natural identification \begin{equation}\label{EqDefPiRtimes}\pi_{\rtimes}: N\rtimes M \cong B(L^2(N)).\end{equation}
\end{Cor} 

Note that by \eqref{EqImplGalAl} and \eqref{EqGalImpCross2} the inverse of the above isomorphism is implemented by the Galois unitary, \begin{equation}\label{EqGalCrossIso} \pi_{\rtimes }^{-1}: B(L^2(N))\cong N\rtimes M,\quad x\mapsto \Ww^*(1\otimes x)\Ww.\end{equation}

If $(N,\alpha)$ is a right coaction, an nsf weight $\psi_{N}$ on $N$ is called \emph{right invariant} if for all states $\omega \in M_*$ and all $x\in N^+$ one has \[\psi_{N}((\id \otimes \omega)\alpha(x)) = \psi_{N}(x).\] 

\begin{Theorem}\cite[Theorem 4.19]{DeC11} If $(N,\alpha)$ is a Galois object, there exists a right invariant nsf weight $\psi_{N}$ for $\alpha$, unique up to multiplication by a positive scalar. 
\end{Theorem}

In general, there is no canonical normalisation available for $\psi_{N}$. Nevertheless, once $\psi_N$ has been chosen there exists by its construction in \cite[Theorem 4.19]{DeC11} and \cite[Lemma 4.18 and Theorem 4.23]{DeC11} a unique $\delta_N>0$ affiliated with $N$ such that \begin{equation}\label{EqConDerPhiPsi}(D\psi_{N} : D\varphi_{N})_t = \nu^{it^2/2} \delta_{N}^{it}\end{equation} with $\delta_{N} \eta N$ an invertible positive operator and $\nu$ the scaling constant of $M$. We further have by \cite[Lemma 4.18]{DeC11} that \begin{equation}\label{EqRelInv}\sigma_t^{\varphi_N}(\delta_N^{is}) = \nu^{ist}\delta_N^{is}\end{equation} and by \cite[Proposition 4.16]{DeC11} that  
\begin{equation}\label{EqDelGrouplike} \alpha(\delta_{N}^{it}) = \delta_{N}^{it}\otimes \delta^{it},\end{equation} with $\delta$ the modular element of $M$.

Evidently, one can develop also a theory of \emph{left} Galois objects $(N,\gamma)$, where $\gamma$ is an integrable ergodic left coaction with the Galois map $\Vv$ determined by \[(\omega \otimes \id)(\Vv)\Gamma_N(x) = \Gamma_N((\omega \otimes \id)\gamma(x)),\qquad \omega \in B(L^2(M))_*\] a unitary, where $\Gamma_N$ is the GNS-map with respect to the weights $\psi_N$ determined by $\psi_N(x) = (\psi\otimes \omega)\gamma(x)$ for all $x\in N^+$ and $\omega\in \mathcal{S}_*(N)$. All of the above results then have their left analogue. 

The main source of examples of Galois objects comes from the theory of \emph{projective corepresentations}.   

\begin{Def}\cite[Definition 7.1 and Theorem 7.2]{DeC11} Let $(M,\Delta)$ be a quantum group von Neumann algebra. A \emph{projective corepresentation} of $(M,\Delta)$ is a coaction $\alpha: N \rightarrow N\vNtimes M$ with $N$ a type $I$-factor. 
\end{Def}

For \[\alpha: N\rightarrow N\vNtimes M\] a projective corepresentation we can form 
\[N_{\alpha} = \alpha(N)' \cap (N\rtimes M),\] and we have a canonical isomorphism of von Neumann algebras \[N\rtimes M \cong N\vNtimes N_{\alpha},\] as $N$ is a type $I$ factor. It is easy to see that the dual right coaction $\hat{\alpha}$ of $\hat{M}'$ restricts to a coaction \begin{equation}\label{EqAlCirc}\alpha^{\circ}: N_{\alpha} \rightarrow N_{\alpha}\vNtimes \hat{M}'.\end{equation}

\begin{Theorem}\label{TheoRightGaloisDual} The couple $(N_{\alpha},\alpha^{\circ})$ is a right Galois object for $(\hat{M}',\hat{\Delta}')$.
\end{Theorem} 

\begin{proof} This is contained in the proof of Theorem 7.2 of \cite{DeC11}.
\end{proof}

In particular, let $\alpha:N\rightarrow N\vNtimes M$ be a right Galois object for $(M,\Delta)$. As $N\rtimes M$ is a type $I$-factor, $\hat{\alpha}$ is a projective corepresentation of $\hat{M}'$. We hence obtain a right Galois object for $(M,\Delta)$ defined by \[((N\rtimes_{\alpha}M)_{\hat{\alpha}},\hat{\alpha}^{\circ}).\]

\begin{Theorem}\label{TheoBidualGal} Let $(N,\alpha)$ be a right Galois object for $(M,\Delta)$. Then we have an isomorphism of von Neumann algebras \[\pi: N \cong (N\rtimes_{\alpha}M)_{\hat{\alpha}}, \quad x\mapsto \chi_{\rtimes}(\Ww^*(x\otimes 1)\Ww)\] which is equivariant in the sense that \begin{equation}\label{EqEqProj} (\pi \otimes \Ad(\Theta))(\alpha(x)) = \hat{\alpha}^{\circ}(\pi(x)),\qquad x\in N.\end{equation}
\end{Theorem}
 
Recall that the isomorphism $\chi_{\rtimes}$ was defined in \eqref{EqDefChiRtimes}.

\begin{proof} This is indirectly contained in \cite{DeC11}, but let us give a direct proof. 

As $\chi_{\rtimes}^{-1}$ puts $N\rtimes M$ is in its ordinary position on $L^2(N)\otimes L^2(M)$, we find \[\chi_{\rtimes}^{-1}( (N\rtimes M)_{\hat{\alpha}})  = (N\vNtimes B(L^2(M))\cap (N\rtimes M)'.\]

Let now $\Ww$ be the Galois unitary for $(N,\alpha)$. It then follows from \eqref{EqImplGalAl}, \eqref{EqWInN}, \eqref{EqGalImpCross2} and the final part of Theorem \ref{TheoIntImp} that \begin{equation}\label{EqConjXX} N\otimes 1 = \Ww((N\vNtimes B(L^2(M))\cap (N\rtimes M)')\Ww^*.\end{equation} This proves that $\pi$ as in the statement of the theorem is a well-defined isomorphism. 

It remains to show \eqref{EqEqProj}, which by \eqref{EqEqBidual} reduces to \[W_{23}(\alpha\otimes \id)(\Ww^*(x\otimes 1)\Ww)W_{23}^*  = \Ww_{13}^*\alpha(x)_{12}\Ww_{13},\qquad x\in N.\] 
But this follows from the fact that $\Ww$ implements $\alpha$, together with the hybrid pentagon equation \eqref{EqHybPent}.
\end{proof} 

Let us single out an important operation which was hidden in the definition of the map $\pi$ in the above theorem.

\begin{Lem} Let $(N,\alpha)$ be a Galois object with Galois unitary $\Ww$. Then \[\Ad_{\alpha}(x) = \Ww^*(x\otimes 1)\Ww\] defines a right coaction of $(\hat{M},\hat{\Delta})$ on $N$.
\end{Lem}

\begin{proof}  By \eqref{EqWInN} and \eqref{EqGalImpCross2}, we have \[\Ww^*(x\otimes 1)\Ww \in N\vNtimes \hat{M},\qquad x\in N.\] Since \[\hat{\Delta}(y)= \Ad(\Sigma W^*\Sigma)(1\otimes y),\qquad y\in \hat{M},\] the coaction property of $\Ad_{\alpha}$ follows straightforwardly from \eqref{EqHybPent}.
\end{proof}

\begin{Def}\label{DefAdj} We call the coaction $\Ad_{\alpha}$ the \emph{adjoint coaction} associated to the Galois object.
\end{Def} 

In the Hopf algebra setting, this is known as the \emph{Miyashita-Ulbrich action}, see \cite{Sch04}. 
  
\section{Co-linking quantum groupoid von Neumann algebras}\label{SecLinCol}

The following considerations can be found in the Hopf algebra setting in \cite{Bic14}, and in the more general setting of measured quantum groupoids in \cite[Section 5]{Eno12}. However, for our particular situation we can slightly relax the necessary conditions in the following definition.

\begin{Def} A \emph{co-linking quantum groupoid von Neumann algebra} will consist of a collection of four (non-zero) von Neumann algebras $Q^{ij}$ together with a collection of eight unital $*$-homomorphisms \[\Delta^{ij;k}: Q^{ij}\rightarrow Q^{ik}\vNtimes Q^{kj}\] satisfying for all indices \begin{equation}\label{EqJointCoass} (\Delta^{ik;l}\otimes \id)\Delta^{ij;k}(x) = (\id\otimes \Delta^{lj;k})\Delta^{ij;l}(x),\quad x\in Q^{ij},\end{equation} and such that on each $Q^{ij}$ there exist nsf weights $\varphi^{ij},\psi^{ij}$ such that for all indices \begin{equation}\label{EqLeftInvG}(\id\otimes \varphi^{jk})\Delta^{ij;k}(x) = \varphi^{ij}(x)1^{ik},\qquad \forall  x\in (Q^{ij})^+,\end{equation} \begin{equation}\label{EqRightInvG} (\psi^{ik}\otimes \id)\Delta^{ij;k}(x) = \psi^{ij}(x)1^{jk},\qquad \forall x\in (Q^{ij})^+,\end{equation}
with $1^{ij}$ the unit in $Q^{ij}$. 
\end{Def} 

We further write \[Q =  Q_{11}\oplus Q_{12}\oplus Q_{21} \oplus Q_{22},\] for the total von Neumann algebra, and \[\Delta_Q: Q\rightarrow Q\vNtimes Q,\quad x \mapsto \sum_{i,j,k} \Delta^{ij;k}(1^{ij}x),\quad x\in Q.\] Note that $Q$ is a non-unital normal $^*$-homomorphism such that \[(\Delta_Q\otimes \id)\Delta_Q = (\id\otimes \Delta_Q)\Delta_Q,\quad \Delta_Q(1^{ij}) = \sum_k 1^{ik}\otimes 1^{kj}.\] 

Consider now the following two faithful unital $*$-homomorphisms from $\C^2$ into $Q$, \begin{equation}\label{EqInca} a: \C^2\rightarrow Q,\quad (r,s) \mapsto r(1^{11}+1^{12}) + s(1^{21}+1^{22}),\end{equation}\begin{equation}\label{EqIncb}b: \C^2\rightarrow Q,\quad (r,s) \mapsto r(1^{11}+ 1^{21}) + s(1^{12}+1^{22}).\end{equation} Then identifying \[L^2(Q) = \oplus_{ij} L^2(Q^{ij}),\quad \xi  =   \oplus \xi^{ij},\] we can view $L^2(Q)$ as a $\C^2$-$\C^2$-bimodule by left multiplication composed with the maps $a$ and $b$. Considering on $\C^2$ the positive functional \[\epsilon: \C^2\rightarrow \C,\quad (r,s)\mapsto r+s,\] so that the GNS-space of $\epsilon$ can be identified as $\C^2$ with the usual Hilbert space structure and standard basis $\{e_i\}$, it then follows that each $\xi^{ij} \in L^2(Q^{ij})$ is left bounded with respect to $(b,\epsilon)$, with \[L_{\epsilon}(\xi^{ij})e_k = \delta_{kj}\xi^{ij},\quad L_{\epsilon}(\xi^{ij})^*\eta^{kl} = \delta_{ik}\delta_{jl}\langle \xi^{ij},\eta^{kl}\rangle e_j.\] It follows that  the \emph{relative tensor product} (or \emph{Connes fusion}) \[L^2(Q) \,{}_{b}\!\underset{\epsilon}{*}\! {}_a\, L^2(Q),\] see for example \cite[Chapter IX.3]{Tak2}, can be identified with the direct sum  $\oplus_{ijk} L^2(Q^{ij})\otimes L^2(Q^{jk})$ via the unitary \[\oplus_{ijk} L^2(Q^{ij})\otimes L^2(Q^{jk}) \rightarrow L^2(Q) \,{}_{b}\!\underset{\epsilon}{*}\! {}_a\, L^2(Q),\quad \xi^{ij}\otimes \eta^{jk} \mapsto \xi^{ij} \underset{\epsilon}{\otimes} \eta^{jk}.\] It follows as well immediately that the \emph{fiber product} $Q \,{}_{b}\!\underset{\C^2}{*}\! {}_a\, Q$, as defined in \cite{Sau83}, can be identified under the above isomorphism as \[ Q \,{}_{b}\!\underset{\C^2}{*}\! {}_a\, Q \cong \oplus_{ijk} Q^{ij}\vNtimes Q^{jk}.\] We see then that $\Delta_Q$ is in fact a \emph{unital} $*$-homomorphism \[\Delta_Q: Q \rightarrow  \oplus_{ijk} Q^{ij}\vNtimes Q^{jk}\subseteq Q\vNtimes Q,\] and that $(Q,\C^2,a,b,\Delta_Q)$ forms a \emph{Hopf bimodule} in the sense of \cite[Definition 2.2.1]{Val96}.

Let now $T_{\varphi},T_{\psi}$ be the respective operator valued weights \[T_{\varphi}: Q^+ \rightarrow a(\C^2),\quad x \mapsto \sum_{ij} \varphi^{ij}(x) a(e_i),\quad x\in Q^+,\]\[T_{\psi}: Q^+ \rightarrow b(\C^2),\quad x\mapsto \sum_{ij}\psi^{ij}(x)b(e_i),\quad x\in Q^+.\] It is trivial to verify that $T_{\varphi}$ and $T_{\psi}$ are respectively \emph{left and right} invariant nsf operator valued weights in the sense of \cite[Definition 3.2]{Les07}. We further write \[\varphi_Q = \epsilon \circ T_{\varphi}  = \oplus_{ij}\varphi^{ij},\quad \psi_Q = \epsilon \circ T_{\psi} = \oplus_{ij} \psi^{ij},\] and by $\Lambda_Q = \oplus_{ij} \Lambda^{ij}$ and $\Gamma_Q = \oplus_{ij} \Gamma^{ij}$ the associated GNS-maps. 

We now obtain immediately from the above and the results of \cite{Les07} the following theorem, noting that we have, in the notation of \cite[Section 3.2]{Les07} \[\hat{b}(x) = J_{Q}a(x)^*J_Q = a(x),\quad x\in \C^2\] and \[L^2(Q) \,{}_a\!\underset{\epsilon}{*}\! {}_a\, L^2(Q) \cong \oplus_{ijk} L^2(Q^{ij})\otimes L^2(Q^{ik}).\]

\begin{Theorem}\cite[Corollary 3.18 and Theorem 3.51]{Les07}\label{TheoMUQGrd} There exists a unique partial isometry \[\mWw \in Q\vNtimes B(L^2(Q))\] with \[\mWw^*\mWw =\sum_{ijk} 1^{ij}\otimes 1^{jk} ,\quad \mWw\mWw^*  = \sum_{ijk} 1^{ij}\otimes 1^{ik},\] and such that for all $\omega \in Q_*$ \[(\omega\otimes \id)(\mWw^*)\Lambda_Q(x) = \Lambda_Q((\omega\otimes \id)\Delta_Q(x)),\quad x\in \mathscr{N}_{\varphi_Q}.\] Moreover, $\mWw$ satisfies the pentagon equation \[\mWw_{12}\mWw_{13}\mWw_{23} = \mWw_{23}\mWw_{12}\] and implements $\Delta_Q$ in the sense that \[\Delta_Q(x) = \mWw^*(1\otimes x)\mWw,\qquad \forall x\in Q.\] 
\end{Theorem}

By restricting $\mWw$, we obtain unitaries \[W^{ij;k}: L^2(Q^{ij})\otimes L^2(Q^{jk}) \rightarrow L^2(Q^{ij})\otimes L^2(Q^{ik})\] satisfying the sixteen \emph{hybrid pentagon equations}
\begin{equation}\label{EqHybPentW}  W_{12}^{ij;k}W_{13}^{ij;l}W_{23}^{jk;l} = W_{23}^{ik;l}W_{12}^{ij;k}
\end{equation}
as maps from $L^2(Q^{ij})\otimes L^2(Q^{jk})\otimes L^2(Q^{kl})$ to $L^2(Q^{ij})\otimes L^2(Q^{ik})\otimes L^2(Q^{il})$. These implement the $\Delta^{ij;k}$ in the sense that \begin{equation}\label{EqPartCom}\Delta^{ij;k}(x) = (W^{ik;j})^*(1\otimes x)W^{ik;j},\quad x\in Q^{ij}.\end{equation}

To show now that $(Q,\Delta_Q)$ fits within the theory of measured quantum groupoids, we still have to show, following the axiom system as in \cite[Definition 3.7]{Eno08}, that the modular automorphism groups of $\varphi_Q$ and $\psi_Q$ commute. Let us prove this by using directly the results of \cite{DeC11}, making contact with the theory of Galois objects.

\begin{Lem} If $(Q,\Delta_Q)$ is a co-linking quantum groupoid von Neumann algebra, then  $(Q^{12},\Delta^{12;2})$ is a right Galois object for $(Q^{22},\Delta^{22;2})$ with Galois unitary $\Ww=W^{12;2}$. Similarly, $(Q^{21},\Delta^{21;1})$ is a right Galois object for $(Q^{11},\Delta^{11;1})$. 
\end{Lem}

\begin{proof} It follows immediately from the definition that $(Q^{22},\Delta^{22;2})$ is a quantum group von Neumann algebra with invariant weights $\varphi^{ii}$ and $\psi^{ii}$, and that $(Q^{12},\Delta^{12;2})$ is an ergodic and integrable right coaction for $(Q^{22},\Delta^{22;2})$. From Theorem \ref{TheoMUQGrd} it follows that $(Q^{12},\Delta^{12;2})$ is in fact a Galois object, with the above unitary as Galois unitary. 

The second statement follows by symmetry.
\end{proof}

Note that from the arguments in the proof of the above lemma, it follows as well that all weights $\varphi^{ij}$ and $\psi^{ij}$ are unique up to positive scalars. In fact, if we  fix the $\varphi^{ii}$ and $\psi^{ii}$, it follows that all $\varphi^{ij},\psi^{ij}$ are uniquely determined. 

\begin{Cor}\label{CorMQGr} Let $(Q,\Delta_Q)$ be a co-linking quantum groupoid. Then $(Q,\Delta_Q)$ defines a measured quantum groupoid in the sense of \cite[Definition 3.7]{Eno08}.
\end{Cor}
\begin{proof} As stated, the only thing left to prove is the commutation between the modular automorphism groups of $\varphi_Q$ and $\psi_Q$. However, these split as the direct sum of the modular automorphism groups of the components $\varphi^{ij}$ and $\psi^{ij}$, so it is enough to prove the commutation for the latter.

As $(Q^{11},\Delta^{11;1})$ and $(Q^{22},\Delta^{22;2})$ are quantum group von Neumann algebras, the commutation follows for the modular automorphisms groups  of the weights $\varphi^{ii}$ and $\psi^{ii}$. Further, by uniqueness as in \cite[Theorem 4.23]{DeC11}, $\varphi^{ij}$ and $\psi^{ij}$ must be the left and right invariant weight associated to the Galois object $(Q^{ij},\Delta^{ij;j})$ for $(Q^{jj},\Delta^{jj;j})$. But we know by \eqref{EqConDerPhiPsi} and \eqref{EqRelInv} that then the modular automorphism groups of $\varphi^{ij}$ and $\psi^{ij}$ commute. 
\end{proof} 

Corollary \ref{CorMQGr} now allows us to use the full strength of the theory of \cite{Les07} and \cite{Eno08}. In particular, we distill the following results which we will need. 

\begin{Theorem}\cite[Theorem B.7]{Eno08} There exists on $Q$ a unique $\sigma$-weakly continuous one-parametergroup of automorphisms $(\tau_t^{Q})_{t\in \R}$, called the \emph{scaling group}, such that \[\Delta_Q\circ \tau_t^{Q} = (\sigma_t^{\varphi_Q}\otimes \sigma_{-t}^{\psi_Q})\circ \Delta_Q.\] Moreover, then $\tau_t^Q$ satisfies also \[\Delta_Q \circ \tau_t^Q = (\tau_t^Q\otimes \tau_t^Q)\circ \Delta_Q.\] 
\end{Theorem}

It follows that $\tau_t^Q$ splits into a direct sum of automorphisms \[\tau_t^Q = \oplus_{ij} \tau_t^{ij}.\] 

\begin{Theorem}\cite[Theorem A.6 and Theorem A.9]{Eno08}\label{TheoEnoUnAn} There exists on $Q$ a unique involutive anti-isomorphism $R_Q$, called the \emph{unitary antipode}, such that, writing $S_Q = R_Q\circ \tau_{-i/2}^Q$, we have for all $\omega \in Q_*$ that \[(\id\otimes \omega)(\mWw) \in \mathscr{D}(S_Q),\quad S_Q((\id\otimes \omega)(\mWw)) = (\id\otimes \omega)(\mWw^*).\] Moreover, $R_Q$ and $\tau_t^Q$ commute, and \[\Delta_Q\circ \tau_t^Q = (\tau_t^Q \otimes \tau_t^Q)\circ \Delta_Q,\quad \Delta_Q\circ R_Q = (R_Q\otimes R_Q)\circ \Delta_Q^{\opp}.\]
\end{Theorem} 

It follows immediately that $R_Q$ splits as a direct sum of anti-isomorphisms \[R_Q = \oplus R^{ij},\quad R^{ij}: Q^{ij}\rightarrow Q^{ji}.\]

It is clear that $\tau^{ii}_t$ and $R^{ii}$ restrict to the scaling group and unitary antipode of the von Neumann algebraic quantum groups $(Q^{ii},\Delta^{ii;i})$. In the following we will then scale also the $\psi^{ij}$ such that \[\psi^{ij} = \varphi^{ji}\circ R^{ij}.\] 

Note now that the \emph{scaling operator} for $(Q,\Delta_Q)$ is a scalar $\nu$ by \cite[Theorem 3.8.(vi)]{Eno08}, and must hence coincide with the scaling constant of both $(Q^{11},\Delta^{11;1})$ and $(Q^{22},\Delta^{22;2})$. It is clear that we can then define a unique anti-unitary $\hat{J}_Q$ such that\footnote{We add the factor $\nu^{i/4}$ to be in consistency with the conventions of the first section and the general duality theory for measured quantum groupoids, but this factor will be irrelevant in what follows.} \[\hat{J}_Q: L^2(Q) \rightarrow L^2(Q),\quad \Lambda_Q(x) \mapsto \nu^{i/4}\Gamma_Q(R_Q(x)^*).\] 

\begin{Theorem}\cite[Theorem 3.10 and Theorem 3.11]{Eno08}\label{TheoConjJhat} One has \[\hat{J}_Qx^*\hat{J}_Q = R_Q(x),\quad \forall x\in Q,\] and \begin{equation}\label{EqJInterW}(\hat{J}_Q\otimes J_Q)\mWw(\hat{J}_Q\otimes J_Q) = \mWw^*.\end{equation}
\end{Theorem}

Note that $\hat{J}_Q$ splits up into the components \[\hat{J}^{ij}: L^2(Q^{ij}) \rightarrow L^2(Q^{ji}).\] Then \eqref{EqJInterW} splits up into the identities \begin{equation}\label{EqConjWpiece}(\hat{J}^{ij}\otimes J^{ik})W^{ij;k} (\hat{J}^{ji}\otimes J^{jk}) = (W^{ji;k})^*.\end{equation}

The following theorem  shows that a Galois object can be completed into a co-linking quantum groupoid von Neumann algebra.

\begin{Theorem}\label{TheoReconQGr} Let $(N,\alpha_N)$ be a Galois object for a quantum group von Neumann algebra $(M,\Delta_M)$. Then there exists a co-linking quantum groupoid von Neumann algebra $(Q,\Delta_Q)$ such that \begin{equation}\label{EqColinkGal} (Q^{22},\Delta^{22;2}) \cong_{\pi^{22}} (M,\Delta_M),\quad (Q^{12},\Delta^{12;2}) \cong_{\pi^{12}} (N,\alpha_N),\end{equation}
with the understanding that \[\Delta_M \circ\pi^{22} = (\pi^{22}\otimes \pi^{22})\circ \Delta^{22;2},\quad \alpha_N \circ \pi^{12} = (\pi^{12}\otimes \pi^{22})\circ \Delta^{12;2}.\]
\end{Theorem} 

\begin{proof} This is a special case of \cite[Theorem 5.12]{Eno12}.
\end{proof}

\begin{Rem} The proof of \cite[Theorem 5.12]{Eno12} roughly goes as follows: first, a structure \emph{dual} to $(Q,\Delta_Q)$ is constructed, an argument which, for quantum group von Neumann algebras, was already done at the end of \cite[Section 5]{DeC11}. Then, using the duality theory for measured quantum groupoids, a candidate for $(Q,\Delta_Q)$ is constructred, and it is then verified by standard techniques that the corners of $(Q,\Delta_Q)$ can indeed be identified with $(N,\alpha_N)$ and $(M,\Delta_M)$ in the above way. It would be interesting to obtain a direct way of proving the above theorem, without passing through duality theory, but we could not find a particularly quick way to achieve this. See however the discussion below. 
\end{Rem} 

In what remains, we want to show how to construct a co-linking quantum groupoid von Neumann algebra directly from a Galois object $(N,\alpha_N)$, without having to pass through duality theory. This will also show immediately that a co-linking quantum groupoid as in Theorem \ref{TheoReconQGr} is essentially uniquely determined. 

Let $O$ be a copy of $N$ with the opposite multiplication, and fix an anti-isomorphism
\[R_N: N \rightarrow O,\qquad R_O  =R_N^{-1}.\] Define $\psi_O$ to be the nsf weight \[\psi_O = \varphi_N\circ R_O\] on $O$, and write its GNS-map as $\Gamma_O$. Let 
\begin{equation}\label{EqHatJ}\hat{J}_N: L^2(N)\rightarrow L^2(O),\quad \Lambda_N(x) \mapsto \nu^{i/4}\Gamma_O(R(x)^*),\quad x\in \mathscr{N}_{\varphi_N},\end{equation} which is a well-defined anti-unitary whose inverse we write $\hat{J}_O$. We have that $\hat{J}_N$ and $\hat{J}_O$ implement $R_N$, \[R_N(x) =  \hat{J}_Nx^*\hat{J}_O,\qquad x\in N.\]

\begin{Lem}\label{LemExtCom} There exists a unique unital and faithful $*$-homomorphism \[\beta_M: M \rightarrow O\vNtimes N\] such that \[\beta_M(x) = (\hat{J}_O\otimes J_N)\Ww(1\otimes J_MxJ_M)\Ww^*(\hat{J}_N\otimes J_N).\] 
 \end{Lem} 
 \begin{proof} Let $(Q,\Delta_Q)$ be a co-linking quantum groupoid von Neumann algebra with $(Q^{22},\Delta^{22;2}) = (M,\Delta_M)$ and $(Q^{12},\Delta^{12;2}) = (N,\alpha_N)$. Then we may take $L^2(O) = L^2(Q^{21})$ and $\hat{J}_N = \hat{J}^{12}$. Since $\Ww = W^{12;2}$, it then follows by Theorem \ref{TheoConjJhat} that $O = Q^{21}$ and by \eqref{EqConjWpiece} that $\beta_M = \Delta^{22;1}$. 
 \end{proof} 

We call $\beta_M$ the \emph{external comultiplication} of $M$. 

Let now \[\gamma_O(x) = \varsigma (R_N\otimes R_M)\alpha(R_O(x)),\quad x\in O.\] Clearly, this is a left coaction of $(M,\Delta_M)$ on $O$. 

\begin{Lem}\label{LemCombet} The following commutations hold: 
\[(\id\otimes \alpha_N)\beta_M = (\beta_M\otimes \id)\Delta_M.\]
\[(\gamma_O\otimes \id)\beta_M = (\id\otimes \beta_M)\Delta_M.\]
\end{Lem} 

\begin{proof} In the setting of the proof of Lemma \ref{LemExtCom}, we have that $\gamma_O =\Delta^{21;2}$ by Theorem \ref{TheoEnoUnAn} and Theorem \ref{TheoConjJhat}.  The commutation relations now follow from the hybrid coassociativity identities \eqref{EqJointCoass}.
\end{proof}

\begin{Def} We define \[P = \{x\in N\vNtimes O\mid (\alpha_N\otimes \id)(x) = (\id\otimes \gamma_O)(x)\}.\]
\end{Def}

Clearly $P$ is a unital von Neumann subalgebra of $N\vNtimes O$.

\begin{Lem} There exists on $P$ a unique unital, normal, coassociative comultiplication \[\Delta_P(x) = (\id\otimes \beta_M\otimes \id)(\alpha_N\otimes \id)(x),\quad x\in P\subseteq N\vNtimes O.\] Moreover, there exist on $N$ and $O$ a left, respectively right coaction of $P$ by the formulas \[ \gamma_N(x) = (\id\otimes \beta_M)\alpha_N(x) \in P\vNtimes N \subseteq (N\vNtimes O)\vNtimes N,\]
\[\alpha_O(x) = (\beta_M\otimes \id)\gamma_O(x) \in O\vNtimes P \subseteq O\vNtimes (N\vNtimes O).\]
\end{Lem} 
\begin{proof} It follows straightforwardly from the coaction property of $\alpha_N$ and $\alpha_O$, the comultiplicativity of $\Delta_M$ and the commutation relations in Lemma \ref{LemCombet} that for all $\omega \in (N\vNtimes O)_*$  and $x\in P\subseteq N\vNtimes O$ one has \[(\omega\otimes \id\otimes \id)\Delta_P(x) \in P,\quad  (\id\otimes \id\otimes \omega)(\Delta_P(x)) \in P.\] Hence $\Delta_P(P)\subseteq P\vNtimes P$, and the coassociativity follows similarly in a straightforward way from the above commutation relations. 

The proof for the existence of the coactions is similar. 
\end{proof}

Let us finally denote by \[\beta_P: P \rightarrow N\vNtimes O\] the identity map.

\begin{Theorem} The von Neumann algebras $P,N,O,M$ together with the eight maps \[\Delta_M,\alpha_N,\gamma_O,\beta_M,\Delta_P,\gamma_N,\alpha_O,\beta_P\] form a co-linking quantum groupoid von Neumann algebra.
\end{Theorem}

\begin{proof} Let us resume the set-up as in the proof of Lemma \ref{LemCombet}. We claim that $Q \cong P\oplus N\oplus O\oplus M$ by the identity maps on the components $N,O,M$ and the map \[\Delta^{11;2}: Q^{11}\rightarrow P.\] Indeed, it follows immediately from the  hybrid coassociativity identities \eqref{EqJointCoass} and the identifications as in the previous lemmas that $\Delta^{11;2}$ has range in $P$. On the other hand, the definition of $\Delta_P$ and the hybrid coassociativity identities \eqref{EqJointCoass} imply that \begin{multline*} \Delta_P(P) \subseteq (\Delta^{11;2}(Q^{11})\vNtimes Q^{12}\vNtimes Q^{21})\cap (Q^{12}\vNtimes Q^{21})\vNtimes \Delta^{11;2}(Q^{11})) \\ = \Delta^{11;2}(Q^{11})\vNtimes \Delta^{11;2}(Q^{11}).\end{multline*}Hence we can define \[P \rightarrow Q^{11}\vNtimes Q^{11},\quad x \mapsto (\Delta^{11;2}\otimes \Delta^{11;2})^{-1}\circ \Delta_P.\] Again the coassociativity relations show that the image of $P$ lands in the set \[\{x\in Q^{11}\vNtimes Q^{11}\mid (\Delta^{11;1}\otimes \id)(x) = (\id\otimes \Delta^{11;1})(x)\} \subseteq Q^{11}\vNtimes Q^{11}.\] However, it is well-known that this latter set is precisely $\Delta^{11;1}(Q^{11})$, see for example \cite[Theorem 2.7]{Vae01}. We obtain directly from this that $\Delta^{11;2}$ must indeed be an isomorphism between $Q^{11;1}$ and $P$, and the resulting isomorphism \[Q \cong P\oplus N\oplus O\oplus M\] clearly intertwines the respective comultiplications.
\end{proof}

\begin{Def} Let $(M,\Delta_M)$ be a quantum group von Neumann algebra, and $(N,\alpha_N)$ a right Galois object for $(M,\Delta_M)$. We call a quantum group von Neumann algebra $(P,\Delta_P)$ the \emph{reflection} of $(M,\Delta_M)$ across $(N,\alpha_N)$ if there exists a quantum groupoid von Neumann algebra $(Q,\Delta_Q)$ with $(Q^{11},\Delta^{11;1})\cong (P,\Delta_P)$,  $(Q^{22},\Delta^{22;2})\cong (M,\Delta_M)$ and $(Q^{12},\Delta^{12;2})\cong (N,\alpha_N)$.
\end{Def} 

It follows from the above discussion that $(P,\Delta_P)$ is uniquely determined up to isomorphism. 

We will in the following interchangingly use the index notation and the more personalized notation for the co-linking quantum groupoid von Neumann algebra associated to a Galois object $(N,\alpha_N)$, \[Q = Q^{11}\oplus Q^{12}\oplus Q^{21}\oplus Q^{22} = P\oplus N\oplus O\oplus M.\] 

Let us end with the following lemma which we will need later on. 

\begin{Lem}\label{LemLegBig} Let $(Q,\Delta_Q)$ be a co-linking quantum groupoid von Neumann algebra. Then for all $i,j,k$ one has that $Q^{ij}$ is the $\sigma$-weakly closed linear span of \begin{equation}\label{EqLeg2}\{(\id\otimes \omega)(\Delta^{ik;j}(x))\mid x\in Q^{ik}, \omega \in (Q^{kj})_*\}.\end{equation}Similarly, $Q^{ij}$ is the $\sigma$-weakly closed linear span of 
\begin{equation}\label{EqLeg1}\{(\omega\otimes \id)(\Delta^{kj;i}(x))\mid x\in Q^{kj}, \omega \in (Q^{ki})_*\}.\end{equation}
\end{Lem} 
\begin{proof} Let $l$ be an arbitrary index. Let us write $A^{ij}$ for the set in \eqref{EqLeg2}. By \cite[Theorem 3.6]{Eno08}, it follows that $Q^{ik}$ contains all elements of the form $(\id\otimes \chi)(W^{ik;l})$ for $\chi \in B(L^2(Q^{kl}),L^2(Q^{il}))_*$. Since \[(\Delta^{ik;j}\otimes \id)(W^{ik;l})  = W^{ij;l}_{13}W^{jk;l}_{23},\] it follows that $A$ contains all elements of the form $(\id\otimes \omega \otimes \chi)(W_{13}^{ij;l}W_{23}^{ik;l})$ for $\omega \in B(L^2(Q^{ik}))_*$ and $\chi \in B(L^2(Q^{kl}),L^2(Q^{il}))_*$. As $W^{ik;l}$ is a unitary, we obtain that \[A \supseteq \{(\id\otimes \omega)(W^{ij;l})\mid \omega \in B(L^2(Q^{jl}),L^2(Q^{il}))_*\}.\] As $l$ was arbitrary, it follows from \cite[Theorem 3.8.(vii)]{Eno08} that $A^{ij}$ is $\sigma$-weakly dense in $Q^{ij}$. 

The $\sigma$-weak density of the set in \eqref{EqLeg1} follows by symmetry.
\end{proof}

\section{Reduction of Galois objects}

\begin{Theorem}\label{TheoRed} Let $(M,\Delta_M)$ be a quantum group von Neumann algebra, and $(M_1,\Delta_{M_1})$ a quantum group von Neumann subalgebra. Let $(N,\alpha_N)$ be a right Galois object for $(M,\Delta_M)$, and denote \[N_1=\{x\in N\mid \alpha_N(x)\in N\otimes M_1\}.\] Then the restriction $\alpha_{N_1}$ of $\alpha_N$ to $N_1$ makes $(N_1,\alpha_{N_1})$ into a right Galois object for $(M_1,\Delta_{M_1})$. Moreover, the co-linking quantum groupoid von Neumann algebra associated to $(N_1,\alpha_{N_1})$ can be realized as a quantum groupoid von Neumann subalgebra of the co-linking quantum groupoid von Neumann algebra associated to $(N,\alpha_N)$.  \end{Theorem}

In particular, the reflection $(P_1,\Delta_1)$ of $(M_1,\Delta_{M_1})$ across $(N_1,\alpha_{N_1})$ is naturally a quantum group von Neumann subalgebra of the reflection $(P,\Delta_P)$ of $(M,\Delta_M)$ across $(N,\alpha_N)$. 

\begin{proof}  First note that $\alpha_{N_1}$ is a right coaction on $N_1$. Indeed, for $x\in N_1$ and $\omega\in M_*$, we have that \[\alpha_N((\id\otimes \omega)\alpha_{N_1}(x))= (\id\otimes \id\otimes \omega)((\id\otimes \Delta_{M_1})\alpha_N(x))\in N\otimes M_1.\] Hence \[\alpha_{N_1}(N_1) \subseteq N_1\vNtimes M_1.\] Since $\alpha_N$ is a coaction and $\Delta_M$ restricts to $\Delta_{M_1}$ on $M_1$, we have that $\alpha_{N_1}$ is a right coaction of $(M_1,\Delta_{M_1})$ on $N_1$. 

Let us use the notation for the co-linking quantum groupoid von Neumann algebra $(Q,\Delta_Q)$ associated to $(N,\alpha_N)$ as in Section \ref{SecLinCol}. Denote $O_1= R_N(N_1)$. Since \[\gamma_O\circ R_N = (R_M\otimes R_N)\circ \alpha_N^{\textrm{op}},\] and $R_M(M_1)=M_1$ by the easy direction in Theorem \ref{TheoSubvn}, we can also characterize $O_1$ as \[O_1=\{z\in O \mid \gamma_O(z)\in M_1\vNtimes O\}.\] Now denote \[\tilde{P}_1=\{z\in N_1\vNtimes O_1\mid (\alpha_N\otimes \id)(z)=(\id\otimes \gamma_O)(z)\},\] and put \[ P_1= \beta_P^{-1}(\tilde{P}_1)\subseteq P.\] Then \[\Delta_{P}(P_1)\subseteq P_1\vNtimes P_1.\] Indeed: applying $\beta_P\otimes \beta_P$ to $\Delta_P(z)$ for $z\in P_1$, and using that \[(\beta_P\otimes \beta_P)\Delta_P = ((\id\otimes \beta_M)\alpha_N\otimes \id),\] we see that \[((\beta_P\otimes \beta_P)\Delta_P)(z) \in N_1\otimes \beta_M(M_1)\vNtimes O_1,\] so we should only check if $\beta_M(M_1)\in O_1\vNtimes N_1$. Since \[(\id\otimes \alpha_N)\beta_M = (\beta_M\otimes \id)\Delta_M,\qquad (\gamma_O\otimes \id)\beta_M = (\id\otimes \beta_M)\Delta_M,\] this condition is fulfilled. It is further also easy to check that we have \[R_P(P_1)\subseteq P_1,\qquad \tau_t^{P}(P_1)\subseteq P_1\] as well, using the commutations between the $\Delta^{ij;k}$, $R_Q$ and $\tau_t^{Q}$, and the fact that $R_M(M_1)=M_1$ and $\tau_{M}(M_1)= M_1$. Using the other direction in Theorem \ref{TheoSubvn}, we conclude that $(P_1,\Delta_{P_1})$ is a quantum group von Neumann subalgebra of $(P,\Delta_{P})$, and in particular is a quantum group von Neumann algebra.

Note now that $\alpha_{N_1}$ is clearly ergodic. We show that it is integrable. By ergodicity, we have a faithful normal weight  on $N_1$ determined by \[\varphi_{N_1}(x)1=(\id\otimes \varphi_{M_1})\alpha_{N_1}(x),\quad x\in N_1.\] We want to show that it is semi-finite. Take $x\in \mathscr{M}_{\varphi_{M_1}}^+$ and $\omega\in \mathcal{S}_*(O_1)$. Then by left invariance of $\varphi_{M_1}$, \begin{eqnarray*}  \varphi_{N_1}((\omega\otimes \id)\beta_{M_1}(x)) &=& ( \id\otimes \varphi_{M_1})((((\omega\otimes \id)\beta_{M_1})\otimes \id)\Delta_{M_1}(x)) \\ &=& \varphi_{M_1}(x) ,\end{eqnarray*} so that $(\omega\otimes \id)\beta_{M_1}(x)$ is integrable for $\varphi_{N_1}$. From this, the semi-finiteness of $\varphi_{N_1}$ follows.\\

We now want to show that $(N_1,\alpha_{N_1})$ is a Galois object. We do this by already constructing the associated co-linking von Neumann algebraic quantum groupoid.

Denote \[Q_1= P_1\oplus O_1\oplus N_1\oplus M_1\subseteq Q.\] It is again easy to check that $\Delta_{Q}(Q_1)\subseteq Q_1\otimes Q_1$, and that $R_{Q}(Q_1)\subseteq Q_1$. Denote by $\Delta_{Q_1}$ the restriction of $\Delta_Q$ to $Q_1$, and by $R_{Q_1}$ the restriction of $R_{Q}$ to $Q_1$. Denote by $\gamma_{N_1}$ the associated coaction $N_1\rightarrow P_1\otimes N_1$ of $P_1$. By symmetry, also $\gamma_{N_1}$ is an ergodic integrable coaction. Denote \[\psi_{N_1}=(\psi_{P_1}\otimes \id)\gamma_{P_1},\qquad \varphi_{O_1}= \psi_{N_1}\circ R_{Q_1}.\] We want to check that the collection $\varphi_{P_1},\varphi_{O_1},\varphi_{N_1}$ and $\varphi_{M_1}$ satisfies the conditions for left invariant nsf weights on a co-linking von Neumann algebraic quantum groupoid. In fact, apart from trivial cases, symmetry allows us to reduce to two cases, namely the left invariance of the weights with respect to $\beta_{M_1}$ and $\gamma_{N_1}$. For $\beta_{M_1}$, the argument has already been given when discussing integrability of $\alpha_{N_1}$. To prove invariance with respect to $\gamma_{N_1}$, choose $\omega\in \mathcal{S}_*(P_1)$, $\tilde{\omega} \in \mathcal{S}_*(N)$ and $x\in \mathscr{M}_{\varphi_{N_1}}^+$. Then \begin{eqnarray*} \varphi_{N_1}((\omega\otimes \id)\gamma_{N_1}(x)) &=& \varphi_{M_1}((\omega\otimes \tilde{\omega}\otimes \id)((\id\otimes \alpha_{N_1})\gamma_{N_1}(x))) \\ &=& \varphi_{M_1}((((\omega\otimes \tilde{\omega})\gamma_{N_1})\otimes \id)\alpha_{N_1}(x)) \\ &=& ((\omega\otimes \tilde{\omega})\gamma_{N_1})(1)\cdot \varphi_{N_1}(x)\\ &=&  \varphi_{N_1}(x).\end{eqnarray*}

Since $R_{Q_1}$ is an anti-multiplicative $^*$-involution flipping the comultiplication, $Q_1$ has the structure of a co-linking quantum groupoid von Neumann algebra. In particular, $(N_1,\alpha_{N_1})$ is a Galois object. By uniqueness, $(Q_1,\Delta_{Q_1})$ will then be a realization of the co-linking quantum groupoid von Neumann algebra associated to $(N_1,\alpha_{N_1})$.
\end{proof}

\begin{Def} In the situation of the above theorem, we call $(N_1,\alpha_{N_1})$ the \emph{reduction} of $(N,\alpha_N)$ to $(M_1,\Delta_{M_1})$. 
\end{Def} 

The following theorem proves a uniqueness property for the reduction of a Galois object. 

\begin{Theorem}\label{TheoRedUn} Let $(N,\alpha)$ be a right Galois object for the quantum group von Neumann algebra $(M,\Delta)$, and let $(M_1,\Delta_1)$ be a quantum group von Neumann subalgebra of $(M,\Delta)$. Let $N_1 \subseteq N$ be a von Neumann subalgebra on which $\alpha$ restricts to a coaction $\alpha_1: N_1\rightarrow N_1\vNtimes M_1$ making $(N_1,\alpha_1)$ into a right $(M_1,\Delta_1)$-Galois object. Then $(N_1,\alpha_1)$ equals the reduction of $(N,\alpha)$ to $(M_1,\Delta_1)$.
\end{Theorem} 

This theorem follows immediately from Theorem \ref{TheoRed} and the following lemma. 

\begin{Lem} Let $(N,\alpha_N)$ be a right Galois object for $(M,\Delta_M)$, and let $N_1\subseteq N$ be a von Neumann subalgebra to which $\alpha_N$ restricts as a coaction \[\alpha_1: N_1\rightarrow N_1\vNtimes M.\] If $(N_1,\alpha_{N_1})$ is a Galois object, then $N_1 = N$. 
\end{Lem} 

\begin{proof} This is a special case of \cite[Proposition 5.13]{Eno12}.
\end{proof}

As an application we show that one can completely determine the co-linking quantum groupoid von Neumann algebra of a Galois object once one knows those for a generating collection of quantum group von Neumann subalgebras of the coacting quantum group von Neumann algebra.

\begin{Theorem}\label{CorGenParts} Let $(N,\alpha_N)$ be a right Galois object for $(M,\Delta_M)$ with reflection $(P,\Delta_P)$. Assume that $\{(M_i,\Delta_{M_i})\mid i = 1,2,\ldots,m\}$ are a collection of quantum group von Neumann subalgebras of $(M,\Delta_M)$ such that $M$ is the $\sigma$-weak closure of the linear span of $M_{1}M_2\cdots M_m$. Then with $(N_i,\alpha_{N_i})$ the reduction of $(N,\alpha_N)$ with respect to $(M_i,\Delta_{M_i})$, and $(P_i,\Delta_{P_i})$ the reflection of $(M_i,\Delta_{M_i})$ across $(N_i,\alpha_{N_i})$ considered as quantum group von Neumann subalgebra of $(P,\Delta_P)$, one has that  $N$, resp. $P$, is the $\sigma$-weak closure of the linear span of the $N_1\cdots N_m$, resp. $P_1\cdots P_m$. 
\end{Theorem} 

\begin{proof}
It follows from Lemma \ref{LemLegBig} that $N$ is the $\sigma$-weak closed linear span of \[ A  =\{(\omega\otimes \id)(\beta_M(x))\mid \omega \in O_*,x\in M\}.\] But by Theorem \ref{TheoRed} and the hypothesis, it follows that $A$ is contained in the $\sigma$-weak closed linear span of $A_1\cdots A_m$ with \[ A_i  =\{(\omega\otimes \id)(\beta_{M_i}(x))\mid \omega \in (O_i)_*,x\in M_i\}.\] As the $A_i$ are $\sigma$-weakly dense in $N_i$, the first part of the theorem follows. 

The statement for $P$ follows similarly, replacing $M$ by $N$, $\beta_M$ by $\gamma_N$ and looking at right slices. 
\end{proof}

\section{$I$-factorial Galois objects}

We now come to the main new topic of this paper: Galois objects which define \emph{at the same time} a projective corepresentation, that is, Galois object structures on $B(\Hsp)$ for some Hilbert space $\Hsp$. In the Hopf algebra context, such Galois objects were introduced in \cite[Definition 5.1]{AEGN02}.

\begin{Def} We call \emph{$I$-factorial (right) Galois object} any right Galois object $(N,\alpha)$ with $N$ a type $I$-factor.
\end{Def}

Our main theorem is the following \emph{duality statement}. Recall the notion of adjoint coaction from Definition \ref{DefAdj}.

\begin{Theorem}\label{TheoDualMain} If $(N,\alpha)$ is a $I$-factorial Galois object for $(M,\Delta)$, then $(N,\Ad_{\alpha})$ is a $I$-factorial right Galois object for $(\hat{M},\hat{\Delta})$. Moreover, \[\Ad_{\Ad_{\alpha}} = \alpha.\]
\end{Theorem} 

The theorem will be proven in two steps, see Theorem \ref{TheoAdjGal} and Theorem \ref{TheoMain}. But we first make the following remark.

\begin{Rem} Starting with a $I$-factorial right Galois object $(N,\alpha)$ for $(M,\Delta)$, one thus has two constructions one can apply: 
\begin{itemize}
\item reflect $(M,\Delta)$ across $(N,\alpha)$ to obtain a left Galois object $(N,\gamma)$ of the reflected quantum group $(P,\Delta_P)$, or, what amounts to the same thing, a right Galois object $(N,\gamma^{\opp})$ for $(P,\Delta_P^{\opp})$, or 
\item take the adjoint coaction to obtain the right Galois object $(N,\Ad_{\alpha})$ for $(\hat{M},\hat{\Delta})$. 
\end{itemize} 
It is not clear how these two operations are related. In general they need not commute, and one could in principle obtain an infinite discrete family of quantum group von Neumann algebras for which $N$ has the structure of a $I$-factorial Galois object. 
\end{Rem} 

We now prove the first part of Theorem \ref{TheoDualMain}.

\begin{Theorem}\label{TheoAdjGal} If $(N,\alpha)$ is a $I$-factorial Galois object with respect to $(M,\Delta)$, then $(N,\Ad_{\alpha})$ is a $I$-factorial right Galois object with respect to $(\hat{M},\hat{\Delta})$.
\end{Theorem}

\begin{proof} As $N$ is a type $I$-factor, we can interpret $\alpha$ as a right projective corepresentation of $(M,\Delta)$. Using the notation from Theorem \ref{TheoRightGaloisDual} and the paragraph following it, we can hence consider the right Galois object $(N_{\alpha},\alpha^{\circ})$ for $\hat{M}'$. As by definition \[N_{\alpha} = \alpha(N)' \cap N\rtimes M,\] with $N$ a type $I$-factor by assumption and $N\rtimes M$ a type $I$-factor by Corollary \ref{CorTypI}, it follows that $N_{\alpha}$ is again a type $I$-factor, and $(N_{\alpha},\alpha^{\circ})$ is a right $I$-factorial Galois object for $\hat{M}'$ by Theorem \ref{TheoRightGaloisDual}.

Now under the isomorphism $\pi_{\rtimes}: N\rtimes M  \cong B(L^2(N))$ as in Corollary \ref{CorTypI}, we can identify $N_{\alpha} \cong N'$. If we now denote for any von Neumann algebra $A$ (in standard form) by $C_A$ the canonical anti-isomorphism \[C_A: A \rightarrow A',\quad x \mapsto J_Ax^*J_A,\quad C_{A'} = C_A^{-1},\] we can transport the above right coaction $\alpha^{\circ}$ of $\hat{M}'$ on $N_{\alpha}$ to a right coaction of $\hat{M}$ on $N$,   \begin{equation}\label{EqDiffCoact}x \mapsto (C_{N'}\otimes C_{\hat{M}'})(\pi_{\rtimes}\otimes \id)(\alpha^{\circ}(\pi_{\rtimes}^{-1}(C_N(x)))),\qquad x\in N,\end{equation} which obviously makes $N$ into a $I$-factorial right Galois object for $\hat{M}$.

Let us prove that the above coaction is equal to $\Ad_{\alpha}$. This is equivalent with proving that \begin{multline*}(\pi_{\rtimes}\otimes \id)\alpha^{\circ}(\pi_{\rtimes}^{-1}(x)) \\= ((J_N\otimes \hat{J})\Ww^*(J_N\otimes \hat{J}))(x\otimes 1)((J_N\otimes \hat{J})\Ww^*(J_N\otimes \hat{J}))^*.\end{multline*}

Using the formula for $\pi_{\rtimes}^{-1}$ in \eqref{EqGalCrossIso} and the defining formula \eqref{EqDualCoaction} for the dual coaction $\hat{\alpha}$, we see that it is sufficient to show that there exists an operator $X\in B(L^2(N)\otimes L^2(M))$ such that \[W_{\Theta,23}\Ww_{12}^*X_{13} = \Ww_{12}^*((J_N\otimes \hat{J})\Ww^*(J_N\otimes \hat{J}))_{23},\] or hence that \begin{equation}\label{EqPentAdCo} \Ww_{12}W_{\Theta,23}^* \Ww_{12}^*((J_N\otimes \hat{J})\Ww^*(J_N\otimes \hat{J}))_{23}  \in B(L^2(N))\vNtimes \C\vNtimes B(L^2(M)).\end{equation}

Using the notation as in Section \ref{SecLinCol}, applying $\Ad(\hat{J}^{12}\otimes J^{12}\otimes \hat{J}^{22})$ to \eqref{EqPentAdCo} and using that \[(\hat{J}^{ij}\otimes J^{ik})\Ww^{ij;k}(\hat{J}^{ji}\otimes J^{jk}) = (\Ww^{ji;k})^*\] by \eqref{EqConjWpiece}, we have that the left hand side of \eqref{EqPentAdCo} becomes 
\[(W^{21;2})_{12}^*W^{22;2}_{23}W^{21;2}_{12}(W^{12;2})_{23}^* = W_{13}^{21;2},\] by the hybrid pentagon equations \eqref{EqHybPentW}, proving \eqref{EqPentAdCo}.
\end{proof}

The second part of Theorem \ref{TheoDualMain} is a bit more involved. If we are only interested in proving $(N,\Ad_{\Ad_{\alpha}})\cong (N,\alpha)$ equivariantly, the proof is not that hard, and can be derived relatively straightforwardly from Theorem \ref{TheoBidualGal}. However, to have an actual equality requires computing the Galois unitary of $\Ad_{\alpha}$.  We need some preparations.

Fix a Hilbert space $\Hsp$ such that $N  = B(\Hsp)$. Then we can identify the standard form of $N$ as \[L^2(N) \cong \Hsp\otimes \overline{\Hsp},\] with $N$ acting in the canonical way on the first component, with \[J_N: \Hsp \otimes \overline{\Hsp} \rightarrow \Hsp\otimes \overline{\Hsp},\quad \xi\otimes \overline{\eta} \mapsto \eta\otimes \overline{\xi}\] and with the self-dual cone $\mathfrak{P}_N$ consisting of the positive Hilbert-Schmidt operators under the canonical embedding $\Hsp \otimes \overline{\Hsp} \hookrightarrow B(\Hsp)$. 

We then have canonically that $N' = 1\otimes B(\overline{\Hsp}) \cong B(\overline{\Hsp})$, and we can identify the GNS-space of $N'$ with $L^2(N) = \Hsp \otimes \overline{\Hsp}$ in such a way that $N'$ acts by its canonical action and such that $J_{N'} = J_N$ and $\mathfrak{P}_{N'} = \mathfrak{P}_N$. 

We now have two natural choices for a standard form of $B(L^2(N))$. On the one hand, one has the \emph{first standard implementation} which is available for any $N$, namely $L^2(B(L^2(N))) \cong  L^2(N)\otimes L^2(N)$ with standard representation \[\pi_{B(L^2(N))}(x) = x\otimes 1,\quad J_{B(L^2(N))}(\xi\otimes \eta) = J_N\eta\otimes J_N\xi,\]and with $\mathfrak{P}_{B(L^2(N))}$ corresponding to the trace class operators in $B(L^2(N))$ under the canonical imbedding \[L^2(N) \otimes L^2(N) \hookrightarrow B(L^2(N)),\quad \xi\otimes \eta\mapsto \xi(J_N\eta)^*.\]
On the other hand, since we also have $B(L^2(N)) = B(\Hsp)\vNtimes B(\overline{\Hsp})$, we can use the tensor product standard construction on $L^2(N) \otimes L^2(N) = (\Hsp\otimes \overline{\Hsp})^{\otimes 2}$ such that \[\pi_{N\vNtimes N'}(x)  = x_{14},\quad  J_{N\vNtimes N'} = J_N\otimes J_N\] and $\mathfrak{P}_{B(L^2(N))}$ the closed positive linear span of elements $\xi\otimes \eta$ with $\xi,\eta\in \mathfrak{P}_{N}$. We will call this the \emph{second standard implementation}. A careful inspection shows that the two standard forms are related by the involutive unitary \begin{multline*} U = \Sigma_{24} \in B((\Hsp\otimes \overline{\Hsp})^{\otimes 2}) = B(L^2(N)\otimes L^2(N)),\\ \Ad(U)\circ \pi_{B(L^2(N))} = \pi_{N\vNtimes N'}.\end{multline*}

Let us also in the following the notation \[\hat{\Vv} = (J_N\otimes J_N)\Ww(J_N\otimes J),\] by analogy with \eqref{EqFormDual}. Recall also again the isomorphism $\pi_{\rtimes}: N\rtimes M \cong B(L^2(N))$ from \eqref{EqRepNcrossM}, which restricts to the representation $\hat{\pi}'$ of $\hat{M}'$ on $L^2(N)$.
 
\begin{Theorem}\label{TheoCompCalcGalAd} Let $\Ww_{\Ad}$ be the Galois unitary for $(N,\Ad_{\alpha})$. Then $\Ww_{\Ad}$ satisfies \begin{equation}\label{EqExprWad} \Ww_{\Ad} = U\hat{\Vv}((\hat{\pi}'\otimes \id)((J\hat{J}\otimes J\hat{J})\hat{W}(\hat{J}J\otimes 1))\end{equation} as a map from $L^2(N)\otimes L^2(M)$ to  $L^2(N)\otimes L^2(N)$. 
\end{Theorem}
\begin{proof} Let $\varphi_N$ be the left invariant nsf weight on $N$ for $\alpha$, and let $\varphi_{\Ad}$ be the left invariant nsf weight for $\Ad_{\alpha}$. Let $\Lambda_N,\Lambda_{\Ad}$ be their respective GNS-maps. Recall the coaction $\alpha^{\circ}$ from \eqref{EqAlCirc}, which is the restriction of the dual coaction $\hat{\alpha}$ to $N_{\alpha} = \alpha(N)'\cap N\rtimes M$, and let $\widetilde{\alpha}^{\circ}$ be the coaction \[\widetilde{\alpha}^{\circ} = (\pi_{\rtimes}\otimes \id)\circ \alpha^{\circ}\circ \pi_{\rtimes}^{-1}\] of $\hat{M}'$ on $N'$, so that, using notation as in \eqref{EqDiffCoact}, \begin{equation}\label{DefAltild} \Ad_{\alpha}(x) = (C_{N'}\otimes C_{\hat{M}'})(\widetilde{\alpha}^{\circ}(C_N(x))),\qquad x\in N.\end{equation} Then $(N',\widetilde{\alpha}^{\circ})$ is clearly a right Galois object with left invariant weight \[\varphi_{\widetilde{\alpha}^{\circ}} = \varphi_{\Ad}\circ C_{N'},\] for which the canonical GNS-map is hence given by \[\Lambda_{N'}(x) = J_N\Lambda_{\Ad}(J_NxJ_N),\qquad x\in \mathscr{N}_{\varphi_{\widetilde{\alpha}^{\circ}}}.\]

Similarly, the GNS-map for the  weight $\varphi_{B(L^2(N))}$ on $B(L^2(N))$ corresponding to the tensor product weight  $\varphi_N \otimes \varphi_{\widetilde{\alpha}^{\circ}}$ on $N\vNtimes N'\cong B(L^2(N))$ is given with respect to the second standard implementation by the tensor product of the corresponding GNS-maps, \[\Lambda_{N\vNtimes N'} = \Lambda_N \otimes \Lambda_{N'}.\] On the other hand, we will denote \[\Lambda_{B(L^2(N))} = U \circ \Lambda_{N\vNtimes N'}\] for the GNS-implementation with respect to the first standard implementation. 

Consider now the crossed product $N\rtimes M$ with its canonical weight $\varphi_{N\rtimes M}$ given by \[\varphi_{N\rtimes M}(x) = \varphi_N(\alpha^{-1}( (\id\otimes \hat{\varphi}')\hat{\alpha}(x))),\qquad x\in (N\rtimes M)^+.\]  One can adapt the results of \cite{Vae01} to the right hand setting to deduce that we can make $L^2(N)\otimes L^2(M)$ a standard form for $N\rtimes M$ in such a way that the GNS-map for $\varphi_{N\rtimes M}$ is given by \[\Lambda_{N\rtimes M}((1\otimes y)\hat{\alpha}(x))  =  \Lambda_{N}(x)\otimes \hat{\Lambda}'(y),\quad y\in x\in \mathscr{N}_{\varphi_{N}},\mathscr{N}_{\hat{\varphi}'},.\] Now as by construction $\varphi_{B(L^2(N))} = \varphi_{N\rtimes M}\circ \pi_{\rtimes}$, we can identify \[L^2(N)\otimes L^2(N) = L^2(B(L^2(N))) \cong L^2(N\rtimes M) \cong L^2(N)\otimes L^2(M).\] Inspecting carefully the commuting square (3.2) of \cite{DeC11}, using the first standard implementation for $B(L^2(N))$, one sees that this identification  is implemented by the unitary \[\Ww^*\Sigma:L^2(N)\otimes L^2(N) \rightarrow L^2(N)\otimes L^2(M).\]

Let now $\Ww_{\widetilde{\alpha}^{\circ}}$ be the Galois unitary for $\widetilde{\alpha}^{\circ}$. By \eqref{DefAltild}, we immediately get \begin{equation}\label{EqWadaltild} W_{\Ad} = (J_N\otimes J_N)\Ww_{\widetilde{\alpha}^{\circ}} (J_N\otimes \hat{J}).\end{equation} 

By definition, $\Ww_{\widetilde{\alpha}^{\circ}}$ satisfies \[\Ww_{\widetilde{\alpha}^{\circ}}^*: L^2(N)\otimes L^2(N)\rightarrow L^2(N)\otimes L^2(M),\]\[\Lambda_{N'}(x)\otimes \Lambda_{N'}(y)  \mapsto (\Lambda_{N'}\otimes \hat{\Lambda}')(\widetilde{\alpha}^{\circ}(y)(x\otimes 1)),\qquad x,y\in \mathscr{N}_{\varphi_{\widetilde{\alpha}^{\circ}}}.\] 

Consider the following maps $Z_1,Z_2,Z_3$, which are clearly well-defined and isometric: 
\[Z_1: L^2(N)\otimes L^2(N) \otimes L^2(N)\rightarrow L^2(N) \otimes L^2(N) \otimes L^2(M),\]\[\Lambda_{B(L^2(N))}(x) \otimes \Lambda_{N'}(y) \mapsto (\Lambda_{B(L^2(N))}\otimes \hat{\Lambda}')((\pi_{\rtimes}\otimes \id)(\hat{\alpha}(\pi_{\rtimes}^{-1}(x)))(y\otimes 1)),\]

\[Z_2: L^2(N) \otimes L^2(M)\otimes L^2(N) \rightarrow L^2(N)\otimes L^2(N)\otimes L^2(M),\]\[\Lambda_N(x) \otimes \hat{\Lambda}'(y)\otimes \Lambda_{N'}(z) \mapsto (\Lambda_{B(L^2(N))}\otimes \hat{\Lambda}')((\hat{\pi}'\otimes \id)(\hat{\Delta}'(y))_{23}(xz\otimes 1)),\]

\[Z_3: L^2(N) \otimes L^2(N) \otimes L^2(M) \rightarrow L^2(M)\otimes L^2(N)\otimes L^2(N),\]\[\Lambda_{B(L^2(N))}(x)\otimes \hat{\Lambda}'(y) \mapsto ( \Lambda_{N\rtimes M}\otimes \hat{\Lambda}')(\hat{\Delta}'(y)_{23}(\pi_{\rtimes}^{-1}(x)\otimes 1)).\]

Then it is easily verified that one has the following commuting squares, where we write $Z_0 =  1\otimes \Ww_{\widetilde{\alpha}^{\circ}}^*$,

\begin{equation}\label{EqnComm1} \xymatrix{
\ar[d]_{U_{12}\Sigma_{23}}L^2(N)\otimes L^2(N)\otimes L^2(N) \ar[r]^{Z_0}  &
L^2(N)\otimes L^2(N)\otimes L^2(M) \ar[d]^{U_{12}} \\
L^2(N)\otimes L^2(N)\otimes L^2(N) \ar[r]_{Z_1} &
L^2(N)\otimes L^2(N)\otimes L^2(M)}
\end{equation}
\begin{equation}\label{EqnComm2} \xymatrix{
L^2(N)\otimes L^2(N)\otimes L^2(N)\ar[r]^{Z_1}  &
L^2(N)\otimes L^2(N)\otimes L^2(M)  \\
L^2(N)\otimes L^2(M)\otimes L^2(N) \ar[u]^{\Sigma_{12}\Ww_{12}}\ar[r]_{Z_2} & L^2(N) \otimes L^2(N)\otimes L^2(M)\ar[u]_{\id}}
\end{equation}

\begin{equation}\label{EqnComm3} \xymatrix{
\ar[d]_{U_{12}\Sigma_{23}}L^2(N)\otimes L^2(M)\otimes L^2(N)\ar[r]^{Z_2}  &
L^2(N)\otimes L^2(N)\otimes L^2(M) \ar[d]^{\Ww^*_{12}\Sigma_{12}} \\
L^2(N)\otimes L^2(N)\otimes L^2(M) \ar[r]_{Z_3} & L^2(N) \otimes L^2(M)\otimes L^2(M) }
\end{equation}

\begin{equation}\label{EqnComm4} \xymatrix{
L^2(N)\otimes L^2(N)\otimes L^2(M)\ar[r]^{Z_3}  &
L^2(N)\otimes L^2(M)\otimes L^2(M)  \\\ar[u]^{\Sigma_{12}\Ww_{12}}
L^2(N)\otimes L^2(M)\otimes L^2(M) \ar[r]_{Z_4}& L^2(N) \otimes L^2(M)\otimes L^2(M) \ar[u]_{\id}}
\end{equation}

where $Z_4 = 1 \otimes ((\hat{J}\otimes \hat{J})\hat{W}^*(\hat{J}\otimes \hat{J}))$. 

Combining the commuting squares \eqref{EqnComm1}, \eqref{EqnComm2}, \eqref{EqnComm3} and \eqref{EqnComm4} together, clearing away $\Sigma_{23}$ and using that $U_{12}$ and $\Ww_{13}^*$ commute since the first leg of the latter lies in $N$, we find 
\begin{equation}\label{EqExpralti} 1\otimes \Ww_{\widetilde{\alpha}^{\circ}}^* = (\Sigma U)_{12}^* \Ww_{12} ((\hat{J}\otimes \hat{J})\hat{W}^*(\hat{J}\otimes \hat{J}))_{23}\Ww_{12}^*\Ww_{23}^*(\Sigma U)_{12}(\Sigma U)_{13}\end{equation} Recall now \eqref{EqWadaltild}. Then multiplying \eqref{EqExpralti} to the left with $J_N\otimes J_N\otimes \hat{J}$ and to the right with $J_N\otimes J_N\otimes J_N$, the left hand side turns into $1\otimes \Ww_{\Ad}^*$. On the other hand,  using that \[U = (J_N\otimes J_N)(\Sigma U)^*(J_N\otimes J_N)\] by a small computation, and using that the second leg of $\hat{\Vv}$ intertwines the standard representation of $\hat{M}'$ with $\hat{\pi}'$ by \eqref{EqGalImpCross2}, \eqref{EqRhatJ} and the identity $J_N\hat{\pi}'(x)^*J_N = \hat{\pi}'(\hat{R}'(x))$ for $x\in \hat{M}'$ from \cite[Lemma 4.9.3]{DeC11}, we find that \[1\otimes
\Ww_{\Ad}^* = U_{12}((\hat{\pi}'\otimes \id)((J\hat{J}\otimes 1))\hat{W}^*(\hat{J}J\otimes \hat{J}J))\hat{\Vv}^*)_{23}U_{12}U_{13}.\] Now using that $U_{12}U_{13} = U_{23}U_{12}$, we obtain by taking adjoints and moving $U_{12}$ to the other side that  \[ U_{12}(1\otimes \Ww_{\Ad})U_{12} = 1 \otimes (U\hat{\Vv}((\hat{\pi}'\otimes \id)((J\hat{J}\otimes J\hat{J})\hat{W}(\hat{J}J\otimes 1))).\] But since the first leg of $\Ww_{\Ad}$ lies in $N$, and $1\otimes N$ commutes with $U$, we obtain the expression for $\Ww_{\Ad}$ in \eqref{TheoCompCalcGalAd}.
\end{proof}

\begin{Theorem}\label{TheoMain} If $(N,\alpha)$ is a $I$-factorial Galois object, then $\Ad_{\Ad_{\alpha}} = \alpha$. 
\end{Theorem} 
\begin{proof} Let us use notation as in Section \ref{SecLinCol}. Then we can write \[\hat{\Vv}= (J^{12} \otimes J^{12})W^{12;2}(J^{12}\otimes J^{22}),\]\[(\hat{\pi}'\otimes \id)((J\hat{J}\otimes J\hat{J})\hat{W}(\hat{J}J\otimes 1)) = (\hat{J}^{21}\otimes J^{22})\Sigma (W^{22;1})^*\Sigma (\hat{J}^{12}\otimes \hat{J}^{22}).\] Hence for $x\in N$, we compute using Theorem \ref{TheoEnoUnAn}, Theorem \ref{TheoConjJhat} and the identities \eqref{EqPartCom} that  \begin{eqnarray*} && \hspace{-0.5cm} \hat{\Vv}((J\hat{J}\otimes J\hat{J})\hat{W}(\hat{J}J\otimes 1))\alpha(x) \\ &&=  (J^{12} \otimes J^{12})W^{12;2}(J^{12}\hat{J}^{21}\otimes 1)\Sigma (W^{22;1})^*\Sigma (\hat{J}^{12}\otimes \hat{J}^{22}) \Delta^{12;2}(x) \\ &&= (J^{12} \otimes J^{12})W^{12;2}(J^{12}\hat{J}^{21}\otimes 1)\\ &&\hspace{1cm}\times \Sigma (W^{22;1})^*\Sigma(R^{12}\otimes R^{22})(\Delta^{12;2}(x))^*  (\hat{J}^{12}\otimes \hat{J}^{22}) \\ &&=  (J^{12} \otimes J^{12})W^{12;2}(J^{12}\hat{J}^{21}\otimes 1)\\ && \hspace{1cm} \times \Sigma (W^{22;1})^*\Sigma (\Delta^{21;2})^{\opp}(R^{12}(x^*))  (\hat{J}^{12}\otimes \hat{J}^{22}) \\ &&= (J^{12} \otimes J^{12})W^{12;2}(J^{12}\hat{J}^{21}\otimes 1)\Sigma (W^{22;1})^*\Delta^{21;2}(R^{12}(x^*))\Sigma   (\hat{J}^{12}\otimes \hat{J}^{22})\\ &&= (J^{12} \otimes J^{12})W^{12;2}(J^{12}\hat{J}^{21}\otimes 1)\Sigma (1\otimes R^{12}(x^*))(W^{22;1})^*\Sigma   (\hat{J}^{12}\otimes \hat{J}^{22}) \\ &&= (J^{12} \otimes J^{12})W^{12;2}(J^{12}\hat{J}^{21}\otimes 1)( R^{12}(x^*)\otimes 1)\Sigma (W^{22;1})^*\Sigma   (\hat{J}^{12}\otimes \hat{J}^{22}) \\ &&= (J^{12} \otimes J^{12})W^{12;2}(J^{12}xJ^{12}\otimes 1)(J^{12}\hat{J}^{21}\otimes 1)\Sigma (W^{22;1})^*\Sigma   (\hat{J}^{12}\otimes \hat{J}^{22}) \\ &&= (x\otimes 1)(J^{12} \otimes J^{12})W^{12;2}(J^{12}\hat{J}^{21}\otimes 1)\Sigma (W^{22;1})^*\Sigma   (\hat{J}^{12}\otimes \hat{J}^{22})\\ &&= (x\otimes 1)\hat{\Vv}((J\hat{J}\otimes J\hat{J})\hat{W}(\hat{J}J\otimes 1)),
\end{eqnarray*} 
where in the third last step we used that $J^{12}xJ^{12} \in (Q^{12})'$ commutes with the first leg of $W^{12;2} \in Q^{12}\vNtimes B(L^2(Q))$. 

Since both legs of $U$ commute with $N$, we conclude from Theorem \ref{TheoCompCalcGalAd} that \[\Ww_{\Ad}\alpha(x) = (x\otimes 1)\Ww_{\Ad}),\] hence $\Ad_{\Ad_{\alpha}} = \alpha$. 
\end{proof} 

We can also use the formula in Theorem \ref{TheoCompCalcGalAd} to find a concrete form for the exterior comultiplication\[\beta^{\Ad}_{\hat{M}}: \hat{M} \rightarrow O\vNtimes N\] on $\hat{M}$. Note that we choose here also $O$ as the opposite of $N$ as a right $\hat{M}$-Galois object, and we keep the same modular conjugation maps $\hat{J}_N$ and $\hat{J}_O$. 

\begin{Prop} The exterior comultiplication $\beta^{\Ad}_{\hat{M}}$ is given by \begin{equation}\label{EqExtAd}\beta^{\Ad}_{\hat{M}}(x) = (\hat{J}_N\otimes J_N)U(\hat{\pi}'(\hat{J}x\hat{J})\otimes 1)U^*(\hat{J}_O\otimes J_N)\end{equation} for $x\in \hat{M}$, where the expression is interpreted as an operator on $L^2(O)\otimes L^2(N)$. 
\end{Prop} 
\begin{proof} By Lemma \ref{LemExtCom}, we have \[\beta_{\hat{M}}^{\Ad}(x) = (\hat{J}_N\otimes J_N)\Ww_{\Ad}(1\otimes \hat{J}x\hat{J})\Ww_{\Ad}^*(\hat{J}_O\otimes J_N).\] We use now the formula for $\Ww_{\Ad}$ in Theorem \ref{TheoCompCalcGalAd}. First, note that \begin{eqnarray*}&& \hspace{-1cm}(J\hat{J}\otimes J\hat{J})\hat{W}(1\otimes \hat{J}x\hat{J})\hat{W}^*(\hat{J}J\otimes \hat{J}J) \\ && = (J\hat{J}\otimes J\hat{J})\hat{W}(J\otimes \hat{J})(1\otimes x)(J\otimes \hat{J})\hat{W}^*(\hat{J}J\otimes \hat{J}J) \\ &&= (\hat{J}\otimes J)\hat{W}^*(1\otimes x)\hat{W}(\hat{J}\otimes J)\\ && = (\hat{J}\otimes J)\hat{\Delta}(x)(\hat{J}\otimes J) \\ && = (1\otimes J\hat{J})\hat{\Delta}'(\hat{J}x\hat{J})(1\otimes \hat{J}J).\end{eqnarray*} 
Further, \begin{eqnarray*} &&\hspace{-0.7cm} \hat{\Vv} (1\otimes J\hat{J})(\hat{\pi}'\otimes\id)(\hat{\Delta}'(\hat{J}x\hat{J})(1\otimes \hat{J}J))\hat{\Vv}^* \\ && = (J_N\otimes J_N)\Ww(J_N\otimes \hat{J})(\hat{\pi}'\otimes\id)(\hat{\Delta}'(\hat{J}x\hat{J})(J_N\otimes \hat{J}))\Ww^*(J_N\otimes J_N) \\ && = (J_N\otimes J_N)\Ww(\hat{\pi}'\otimes\id)((J\otimes \hat{J})(\hat{\Delta}'(\hat{J}x\hat{J})(J\otimes \hat{J})))\Ww^*(J_N\otimes J_N) \\ &&= J_N\hat{\pi}'(J\hat{J}x\hat{J}J)J_N\otimes 1 \\ &&= \hat{\pi}'(\hat{J}x\hat{J})\otimes 1,
\end{eqnarray*}
where in the penultimate step we used \cite[Lemma 4.2.(iv)]{DeC11}. By Theorem \ref{TheoCompCalcGalAd} we then find the expression \eqref{EqExtAd}.
\end{proof}

Note that we can take $O = B(\overline{\Hsp})$ with $\overline{\Hsp}$ the conjugate Hilbert space and $R_N$ the transpose map, \[R_N(x) = \overline{x}^* = J_{\Hsp}x^*J_{\Hsp}^*\] for \[J_{\Hsp}: \Hsp \rightarrow \overline{\Hsp},\quad \xi\mapsto \overline{\xi}.\]  Since the GNS-spaces $L^2(N)$ and $L^2(O)$ are then respectively $\Hsp \otimes \overline{\Hsp}$ and $\overline{\Hsp}\otimes \Hsp$, we can identify the map $\hat{J}_N$ with \[\hat{J}_N = \nu^{i/4}(J_{\Hsp}\otimes J_{\overline{\Hsp}}): \Hsp \otimes \overline{\Hsp} \rightarrow \overline{\Hsp}\otimes \Hsp.\] Since $J_N = \Sigma \circ (J_{\Hsp}\otimes J_{\Hsp}^*)$, we can view $\beta_{\hat{M}}^{\Ad}$ as the $*$-homomorphism \[\hat{\beta}_{\hat{M}}(x) = (J_{\Hsp}\otimes J_{\Hsp}^*)\hat{\pi}'(\hat{J}x\hat{J})(J_{\Hsp}^*\otimes J_{\Hsp}) \in B(\overline{\Hsp}) \vNtimes B(\Hsp) = O\vNtimes N,\] where we see $\hat{\pi}'$ as a representation of $\hat{M}'$ on $\Hsp\otimes \overline{\Hsp} = L^2(N)$. 

Let us end by giving some more information concerning the invariant weights on a $I$-factorial Galois object $(N,\alpha)$. 

\begin{Theorem} Let $(N,\alpha)$ be a $I$-factorial right Galois object for $(M,\Delta)$, say $N = B(\Hsp)$ for a Hilbert space $\Hsp$. Let $h$ be the unique positive (unbounded, invertible) operator on $\Hsp$ such that the left invariant weight $\varphi_N$ for $\alpha$ is given by \[\varphi_N(x) = \Tr(h^{1/2}xh^{1/2}),\quad \forall x\in N^+.\] Then the left invariant weight $\varphi_{\Ad}$ for $(N,\Ad_{\alpha})$ is given by \[\varphi_{\Ad}(x) = \Tr(h^{-1/2}xh^{-1/2}),\qquad x\in N^+.\]  
\end{Theorem}

\begin{proof} Let as before \[\varphi_{N\rtimes M}: (N\rtimes_{\alpha} M)^+ \rightarrow [0,+\infty],\quad x\mapsto \varphi_N \circ \alpha^{-1}\circ (\id\otimes \id\otimes \hat{\varphi}')\hat{\alpha}(x)\] be the dual weight of $\varphi_N$ \cite[Definition 3.1]{Vae01}, and denote again \[\varphi_{B(L^2(N))} = \varphi_{N\rtimes M}\circ \pi_{\rtimes}^{-1}.\] Then it follows from \cite[Proposition 3.7]{DeC11} that \[\varphi_{B(L^2(N))} = \Tr(\nabla_N^{1/2}\,\cdot\,\nabla_N^{1/2}).\] With $h$ as in the statement of the theorem, we have however \[\nabla_N^{it} = h^{it}\otimes \overline{h^{it}}  = h^{it}\otimes(\overline{h})^{-it},\] where we write $\overline{x}\overline{\xi} = \overline{x\xi}$ for $x\in B(\Hsp)$ and $\overline{\xi}\in \overline{\Hsp}$ the conjugate of $\xi\in \Hsp$. Hence on $B(\Hsp)\vNtimes B(\overline{\Hsp}) \cong N \vNtimes N' \cong B(L^2(N))$, the above weight can be expressed as \[\varphi_{B(L^2(N))} = \Tr(h^{1/2}\,\cdot\,h^{1/2}) \otimes \Tr((\overline{h})^{-1/2}\,\cdot\,(\overline{h})^{-1/2}).\]

Now by the proof of Theorem \ref{TheoAdjGal} we have that \[\Ad_{\alpha}(x) = (J_N\otimes \hat{J})\widetilde{\alpha}^{\circ}(J_NxJ_N)(J_N\otimes \hat{J}),\] where \[\widetilde{\alpha}^{\circ}(x) = (\pi_{\rtimes}\otimes \id)\alpha^{\circ}(\pi_{\rtimes}^{-1}(x)),\qquad x\in N'.\] As $\alpha^{\circ}$ is the restriction of $\hat{\alpha}$ to the factor $ \alpha(N)'\cap N\rtimes M \cong N'\cong B(\overline{\Hsp})$ which splits of, it  follows immediately from the above discussion that the left invariant weight for $\widetilde{\alpha}^{\circ}$ must be \[x \mapsto \Tr((\overline{h})^{-1/2} x (\overline{h})^{-1/2}),\qquad x\in B(\overline{\Hsp}).\] From the above form for $\Ad_{\alpha}$, we then deduce that \[\varphi_{\Ad} = \Tr(h^{-1/2}\,\cdot\,h^{-1/2}).\] 
\end{proof}

Note that the inversion $h \rightarrow h^{-1}$ in the above theorem is not unexpected: if we rescale $\varphi \rightarrow \lambda\varphi$ for $\lambda>0$, then the dual weight gets rescaled in the inverse way, $\hat{\varphi} \rightarrow \lambda^{-1}\hat{\varphi}$.

\section{Example: Heisenberg double}

Let $(M,\Delta)$ be a quantum group von Neumann algebra, and consider the tensor product quantum group von Neumann algebra $\widetilde{M} = \hat{M}\vNtimes M$ with comultiplication \[\Delta_{\widetilde{M}}(x) =  \varsigma_{23}(\hat{\Delta}\otimes \Delta)(x).\] Then one obtains a right coaction of $\widetilde{M}$ on $\widetilde{N} := B(L^2(M))$ by \[\alpha_{\widetilde{N}}(x)  = \hat{V}_{12}V_{13} (x\otimes 1\otimes 1)V_{13}^*\hat{V}_{12}^*.\] Indeed, since $V\in M^{\wedge \prime}\vNtimes M$ and $\hat{V}\in M^{\prime}\vNtimes \hat{M}$, the map $\alpha_{\widetilde{N}}$ restricts to the coaction $\Delta$ of $M$ on $M$ and the coaction $\hat{\Delta}$ of $\hat{M}$ on $\hat{M}$. Since $M$ and $\hat{M}$ generate $B(L^2(M))$ (see for example \cite[Proposition 2.5]{VV03}), this is sufficient to conclude that $\alpha_{\widetilde{N}}$ is well-defined. This is a generalization of the \emph{Heisenberg algebra} for the Cartesian product of an abelian compact group with its Pontryagin dual. In general we call $(\widetilde{N},\alpha_{\widetilde{N}})$ the \emph{Heisenberg double} of $(M,\Delta)$. 

\begin{Prop} The coaction $(\widetilde{N},\alpha_{\widetilde{N}})$ is a $I$-factorial Galois object.
\end{Prop}

\begin{proof} The $\sigma$-weak closure of the first leg of $\hat{V}_{12}V_{13}$ contains $M' M^{\wedge \prime} = J\hat{J}M\hat{M}\hat{J}J$, and is hence $\sigma$-weakly dense in $B(L^2(M))$. As any coinvariant element in $B(L^2(M))$ commutes with the first leg of $\hat{V}_{12}V_{13}$, the coaction $\alpha_{\widetilde{N}}$ is ergodic.

We deduce easily that $\alpha_{\widetilde{N}}$ is integrable, as all elements of the form $y^*x^*xy$ for $x\in \mathscr{N}_{\varphi}$ and $y \in \mathscr{N}_{\hat{\varphi}}$ are integrable. 

Now the crossed product by $\alpha_{\widetilde{N}}$ is generated by $\Delta(M)_{13}, 1\otimes 1\otimes M^{\wedge \prime}, \hat{\Delta}(\hat{M})_{12}$ and $1\otimes M^{\prime}\otimes 1$. Applying $\Ad(\Sigma_{23}\Sigma_{13}W_{13})$, we obtain that it is isomorphic to the von Neumann algebra generated by \[(M\otimes 1\otimes 1)\cup (M^{\wedge\prime}\otimes 1\otimes 1)\cup (\hat{\Delta}\otimes \id)\hat{\Delta}(\hat{M}) \cup 1 \otimes 1\otimes M'.\] But $M M^{\wedge\prime} = \hat{J}M \hat{M}\hat{J}$ is $\sigma$-weakly dense in $B(L^2(M))$. In particular, we can throw in another copy of $\hat{M}\otimes 1\otimes 1$ in the first leg. Using coassociativity of $\hat{\Delta}$ and the fact that $(\hat{M}\otimes 1)\hat{\Delta}(\hat{M})$ is $\sigma$-weakly dense in $\hat{M}\vNtimes \hat{M}$, we obtain that the above von Neumann algebra is the same as the one generated by  \[B(L^2(M))\otimes 1\otimes 1 \cup 1\otimes \hat{\Delta}(\hat{M}) \cup 1 \otimes 1\otimes M'.\] Applying $\Ad(\Sigma_{23}\hat{W}_{23})$, we see that this becomes an isomorphic copy of the von Neumann algebra $B(L^2(M))\vNtimes B(L^2(M))$. In other words, the crossed product  $\widetilde{N}\rtimes \widetilde{M}$ is a type $I$-factor, which is sufficient to conclude that $(\widetilde{N},\alpha_{\widetilde{N}})$ is a Galois object by Corollary \ref{CorTypI}.
\end{proof}

\begin{Prop} The invariant weights on $\widetilde{N}$ are given by \[\varphi_{\widetilde{N}}(x) = \Tr(h^{1/2}xh^{1/2}),\qquad \psi_{\widetilde{N}}(x) = \Tr(k^{1/2}xk^{1/2}),\] where $h,k$ are the positive invertible operators such that \[h^{it} = \nu^{it^2/2}\nabla^{it}J\delta^{it}J,\quad k^{it} = P^{-it},\] with $P^{it}$ determined by  $P^{it}\Lambda(x) = \nu^{t/2}\Lambda(\tau_t(x))$ for $x\in \mathscr{N}_{\varphi}$. 
\end{Prop} 

\begin{proof} The formula for $\varphi_{\widetilde{N}}$ follows from \cite[Proposition 2.8 and Proposition 2.9]{VV03}. By \cite[Theorem 4.19]{DeC11}, $\psi_{\widetilde{N}} = \varphi_{\widetilde{N}}(\delta_{\widetilde{N}}^{1/2}\,\cdot\,\delta_{\widetilde{N}}^{1/2})$ with the positive operator $\delta_{\widetilde{N}}$ determined up to a positive scalar by the fact that \[\alpha_{\widetilde{N}}(\delta_{\widetilde{N}}^{it}) = \delta_{\widetilde{N}}^{it}\otimes \hat{\delta}^{it}\otimes \delta^{it}.\] It follows that we can take \[\delta_{\widetilde{N}}^{it} = \nu^{it^2/2} \hat{\delta}^{it}\delta^{it},\] and the form for $\psi_{\widetilde{N}}$ now follows from the remarks above \cite[Proposition 2.10]{VV03}.
\end{proof} 

Let us now compute the Galois unitary associated to $(\widetilde{N},\alpha_{\widetilde{N}})$. In fact, following the discussion after \cite[Proposition 2.9]{VV03}, we may realize $\Lambda_{\widetilde{N}}$ as the unique GNS-map having the linear span of $\mathscr{N}_{\hat{\varphi}}\mathscr{N}_{\varphi}$ as its core and on which \[\Lambda_{\widetilde{N}}(xy) = \hat{\Lambda}(x)\otimes \Lambda(y).\] The corresponding GNS-representation of $B(L^2(M))$ is by \[x \rightarrow V(x\otimes 1)V^*,\] which identifies $B(L^2(M))$ with $M\ltimes_{\Delta}M$. Taking this presentation, we can represent the Galois unitary $\widetilde{\Ww}$ of $(\widetilde{N},\alpha_{\widetilde{N}})$ as a unitary operator on $L^2(M)^{\otimes 4}$. 

\begin{Prop} The Galois unitary of $(\widetilde{N},\alpha_{\widetilde{N}})$ is given by 
\[\widetilde{\Ww} = W_{14}W_{24} W_{31}^*.\]
\end{Prop} 
\begin{proof} Using KSGNS-maps of the form $\id\otimes \Lambda$ and identities of the form \[(\id\otimes \Lambda)(\Delta(y)) = W^*(\id\otimes \Lambda(y)),\qquad y\in \mathscr{N}_{\varphi},\] we compute for $x\in \mathscr{N}_{\hat{\varphi}}$, $y\in \mathscr{N}_{\varphi}$ and $\xi,\eta\in L^2(M)\otimes L^2(M)$ that \begin{eqnarray*} \hspace{-0.6cm} && \widetilde{\Ww}^*((\xi\otimes \eta) \otimes \Lambda_{\widetilde{N}}(xy))\\ && = ((\id\otimes \id\otimes \hat{\Lambda})(\hat{\Delta}(x)_{13})\otimes \id)(\id\otimes \id\otimes \Lambda_M)((\Delta\otimes \id)\Delta(y))(\xi\otimes \eta) \\ && = ((\id\otimes \id\otimes \hat{\Lambda})(\hat{\Delta}(x)_{13})\otimes \id) W_{23}^*W_{13}^* (\xi\otimes \eta \otimes \Lambda(y)) \\ &&= \hat{W}_{13}^*W_{24}^* W_{14}^* (\xi\otimes \eta\otimes \hat{\Lambda}(x)\otimes \Lambda(y)) \\ &&= \hat{W}_{13}^*W_{24}^* W_{14}^* ((\xi\otimes \eta)\otimes \Lambda_{\widetilde{N}}(xy)).
\end{eqnarray*} 
\end{proof}

To have an expression for $\widetilde{\Ww} \in \widetilde{N}\vNtimes B(L^2(\widetilde{M}),L^2(\widetilde{N}))$ with the first leg in its ordinary representation on $B(L^2(M))$, note that \begin{multline*} V_{12}^*W_{14}W_{24}W_{31}^*V_{12} = V_{12}^*(\Delta\otimes \id)(W)_{124}W_{31}^*V_{12}\\ = W_{14}V_{12}^*W_{31}^*V_{12} = W_{14}W_{31}^*,\end{multline*}  
so we obtain the natural expression \[\widetilde{\Ww} = W_{13}\hat{W}_{12} \in \widetilde{N}\vNtimes B(L^2(\widetilde{M}),L^2(\widetilde{N})).\]

Identifying the dual of $\widetilde{M}$ with $M\vNtimes \hat{M}$ in the natural way, we get the following corollary. 

\begin{Cor} The adjoint coaction $\Ad_{\alpha_{\widetilde{N}}}$ of $M\vNtimes \hat{M}$ on $B(L^2(M))$ is given by \[x \mapsto \hat{W}_{12}^*W_{13}^*(x\otimes 1\otimes 1)W_{13}\hat{W}_{12}.\] 
\end{Cor} 

Using that \[(\id\otimes \varsigma)\alpha_{\widetilde{N}}(x) = V_{12}\hat{V}_{13}(x\otimes 1\otimes 1)\hat{V}_{13}^*V_{12}^*,\] together with the fact that \[\hat{W}_{12}^*W_{13}^* = (\hat{J}J\otimes 1\otimes 1)V_{12}\hat{V}_{13}(J\hat{J}\otimes 1\otimes 1),\] we deduce that \[\Ad_{\alpha_{\widetilde{N}}}(x) = (\Ad(\hat{J}J)\otimes \varsigma_{23})\alpha_{\widetilde{N}}(\Ad(J\hat{J})(x)),\] so that identifying $M\vNtimes \hat{M}\cong \hat{M}\vNtimes M$ via the flip map, $\Ad_{\alpha_{\widetilde{N}}}$ is just an isomorphic copy of $\alpha_{\widetilde{N}}$ itself. 

Let us consider now the reflection of $(\widetilde{M},\Delta_{\widetilde{M}})$ across $(\widetilde{N},\alpha_{\widetilde{N}})$. In fact, we have that $(M,\Delta)$ and $(\hat{M},\hat{\Delta})$ are quantum group von Neumann subalgebras of $(\widetilde{M},\Delta_{\widetilde{M}})$, and it is clear by Theorem \ref{TheoRedUn} that their reductions are respectively $M$ and $\hat{M}$ as sitting inside $B(L^2(M))$. Define \[\widetilde{O} = B(L^2(M)),\quad R_{\widetilde{N}}(x) = \hat{J}x^*\hat{J},\] so that \[\gamma_{\widetilde{O}}(x) = \Delta(x), \quad \gamma_{\widetilde{O}}(y) = \varsigma(J\otimes \hat{J})\hat{\Delta}(\hat{J}y\hat{J})(J\otimes \hat{J}),\quad x\in M,y\in M^{\wedge \prime}.\] It follows then by Theorem \ref{CorGenParts} that the reflected quantum group von Neumann algebra $(\widetilde{P},\Delta_{\widetilde{P}})$ is the $\sigma$-weak closure of $P\hat{P}$ where \[P = \Delta(M),\quad \hat{P} = (1\otimes J\hat{J})\hat{\Delta}(\hat{M})(1\otimes J\hat{J}),\] with the coproduct on $P$ and $\hat{P}$ inherited respectively from $(M,\Delta)$ and $(\hat{M},\hat{\Delta})$. 

We show that $(\widetilde{P},\Delta_{\widetilde{P}})$ is isomorphic to the \emph{Drinfel'd double} $(D(M),\Delta_{D(M)})$ of $(M,\Delta_M)$. We will follow here the approach of \cite[Section 8]{BV05}, and define the Drinfel'd double of $(M,\Delta_M)$ as the dual of the \emph{Drinfel'd codouble} $(\hat{D}(M),\Delta_{\hat{D}(M)})$, which consists of the von Neumann algebra $\hat{D}(M) = M\vNtimes \hat{M}$ endowed with the coproduct \[\Delta_{\hat{D}(M)} = (\id\otimes \varsigma\circ \Ad(W)\otimes \id)\circ (\Delta\otimes \hat{\Delta}).\] 

\begin{Theorem} There is an isomorphism $(D(M),\Delta_{D(M)}) \cong (\widetilde{P},\Delta_{\widetilde{P}})$. 
\end{Theorem} 

\begin{proof}
By \cite[Theorem 5.3 and Proposition 8.1]{BV05}, $D(M)$ can be realized as the von Neumann algebra \[D(M) = ((\hat{M}\otimes 1)\cup (J\otimes \hat{J})W^*(J\otimes \hat{J})W^*(1\otimes M)W(J\otimes \hat{J})W(J\otimes \hat{J}))''\] with the coproduct $\Delta_{D(M)}$ restricting to the coproducts on the copies of $(M,\Delta)$ and $(\hat{M},\hat{\Delta})$ inside. However, conjugating with $(J\otimes \hat{J})W(J\otimes \hat{J})$ shows that \begin{eqnarray*} D(M) &\cong & ((J\otimes \hat{J})W(J\hat{M}J\otimes 1)W^*(J\otimes \hat{J})\cup W^*(1\otimes M)W)''\\ &= & ((J\otimes \hat{J})\Sigma \hat{W}^*(1\otimes \hat{R}(\hat{M}))\hat{W}\Sigma(J\otimes \hat{J})\cup \Delta(M))'' \\ &=& ((J\otimes \hat{J})\hat{\Delta}^{\opp}(\hat{R}(\hat{M})) (J\otimes \hat{J})\cup \Delta(M))'' \\ &=& ((J\otimes \hat{J})(\hat{R}\otimes \hat{R})\hat{\Delta}(\hat{M}) (J\otimes \hat{J})\cup \Delta(M))''  \\ &=& ((1\otimes \hat{J}J)\hat{\Delta}(\hat{M})(1\otimes J\hat{J}) \cup \Delta(M))'' \\ &=& \widetilde{P}.
\end{eqnarray*} 
\end{proof}

\end{document}